\documentclass[12pt]{article}
\usepackage{amssymb}
\usepackage{amsfonts}
\usepackage{amsmath}
\usepackage{amsmath}
\usepackage{verbatim}
\usepackage{color}
\usepackage[active]{srcltx}
\include{srctex}
\allowdisplaybreaks[4]
\newtheorem{theorem}{Theorem}[section]

\newtheorem{prop}[theorem]{Proposition}

\newtheorem{corollary}[theorem]{Corollary}

\newtheorem{definition}[theorem]{Definition}

\newtheorem{lemma}[theorem]{Lemma}

\newtheorem{problem}[theorem]{Problem}

\newtheorem{remark}[theorem]{Remark}

\let\Section=\section
\def\section{\setcounter{equation}{0}\Section}
\newenvironment{proof}[1][Proof]{\textbf{#1.} }{\ \rule{0.5em}{0.5em}}
\newcommand{\R}{\mathbb{R}}

\def\RR{\mathbb{R}}

\def\EE{\mathbb{E}}

\def\si{{\sigma}}

\def\si{{\sigma}}
\def\Si{{\Sigma}}

\def \eref#1{\hbox{(\ref{#1})}}

\begin{document}

\title{On  H\"older continuity of the solution of stochastic wave equations }

\author{Yaozhong {\sc Hu}\thanks{Y.  Hu is
partially supported by a grant from the Simons Foundation
\#209206.}, \  Jingyu {\sc Huang} \ and \  David
{\sc Nualart}\thanks{ D. Nualart is supported by the
NSF grant DMS1208625. \newline
      }  \\
Department of Mathematics \\
University of Kansas \\
Lawrence, Kansas, 66045 USA}
%\date{June 28, 2013}
\date{}
\maketitle

\begin{abstract}
In this paper, we study the  stochastic wave equations in the spatial dimension 3 driven by a Gaussian noise which is  white in time and correlated in space. Our main concern is the  sample path  H\"older continuity of the solution both in time variable and in space
 variables. The   conditions are given either in terms of  the mean H\"older continuity of the covariance function or in terms of its spectral measure.   Some examples of the  covariance functions are proved to satisfy our conditions, which include the case of the work \cite{dalang4}.  In particular, we obtain the H\"older continuity results for the solution of the stochastic wave equations  driven by (space inhomogeneous) fractional Brownian noises. For this particular noise,  the optimality of the obtained H\"older exponents
is also discussed.
\end{abstract}

\section{Introduction}
We shall study  the following  stochastic wave equation in spatial dimension
$d=3$:
\begin{eqnarray}\label{our SPDE}
 \left\{\begin{array}{rcl}
&& \left(\frac{\partial ^2}{\partial t^2}  -\Delta  \right) u(t,x)
 = \si(t,x,u(t,x)) \dot W(t,x)  +b(t,x, u(t,x)),  \\ \\
&&u(0, x)=v_0(x),\qquad \frac{\partial u}{\partial t}(0, x)=\bar
v_0(x),
\end{array}
\right.
\end{eqnarray}
where $t\in(0,T]$ for some fixed $T>0$, $x \in \mathbb{R}^3$ and
$\Delta=\frac{\partial ^2}{\partial x_1^2} +\frac{\partial ^2}{\partial x_2^2} +\frac{\partial ^2}{\partial x_3^2} $ denotes the Laplacian on $\mathbb{R}^3$. The coefficients
$\sigma$ and $b$ satisfy some regularity conditions which will be specified later. The Gaussian noise
process $\dot{W}$ is  assumed  to be white in time and with a homogenous correlation in space.   This means
 \[
\EE \left[\dot W(t,x)\dot W(s, y)\right]=\delta(t-s)  f(x-y)
\]
for  a non-negative, non-negative definite and locally integrable function function  $f$,  where $\delta$ is the Dirac delta function.

It is known (see, for instance, \cite[Theorem
4.3]{dalang5})   that if $\sigma$ and $b$ are Lipschitz functions with linear growth and $f$ satisfies $\int_{|x| \le 1} f(x)/|x| dx<\infty$, then there is a unique mild solution to Equation (\ref{our SPDE}).
Our purpose is to establish the   sample path  H\"older continuity both in time variable and in space
 variables   of
the solution to this equation. When   $f$ is given by a Riesz kernel $|x|^{-\beta}$, $\beta \in (0,2)$,
   the H\"older continuity of the solution  has been  obtained
by Dalang and Sanz-Sol\'e in their  monograph \cite{dalang4}.  Their approach is
based on   the  fractional Sobolev imbedding theorem and the Fourier transformation technique.

In this paper,  we shall consider more general Gaussian noises,   and we introduce a new approach  that avoids the Fourier transform. The main idea is to  impose conditions on the covariance $f$ itself.    To be more precise,   let $D_w f=f(\cdot +w)$  be the shift operator.  We shall show that if  $\|D_w f- f\|_{L^1(\rho)}\leq C |w|^{\gamma}$ and $\|D_w f +D_{-w}f -2f\|_{L^1 (\rho)}\leq C |w|^{\gamma^{\prime}}$ for some $\gamma \in (0,1]$ and $\gamma^{\prime} \in (0,2]$, where $\rho$ is the  measure on $\RR^3$ defined to be $\rho(dz)={\bf 1}_{\{|z|\leq 2T\} }\frac{1}{|z|}dz$, then the solution to \eref{our SPDE} is locally   H\"older continuous  of order $\kappa < \min(\gamma, \frac {\gamma'}2)$ in the space variable  (assuming zero initial conditions)  (see Theorem \ref{space Holder Thm no Fourier}).

The H\"older continuity in the time variable is more involved. Following the methodology used by Dalang and Sanz-Sol\'e in \cite{dalang4}, we transform the  time increments into space increments, and we impose suitable  assumptions on the modulus of continuity of  a shift operator which are formulated integrals over $[0,T]\times (S^2)^2$, equipped
with the measure $ds \sigma(d\xi)\sigma(d\eta)  $, where $\sigma$ is the uniform measure in the unit sphere $S^2$
  (see  Theorem \ref{time holder Thm}).

We also obtain a  theorem on the H\"older continuity in  the space variable using  the Fourier transform technique.  More precisely, we establish the  H\"older continuity of order $\kappa <\gamma$, provided the spectral measure $\mu$ satisfies the integrability condition $\int_{\mathbb{R}^3}   \frac { \mu(d\zeta)} { 1+|\zeta|^{2-2\gamma}}  <\infty$ and  the Fourier transform of $|\zeta|^{2\gamma} \mu(d\zeta)$ is  non-negative. The non-negativity condition of this measure  leads to a simple proof of the H\"older continuity in the space variable which avoids  the control of the norms of the increments
$D_w f- f$ and  $D_w f +D_{-w}f -2f$ (or their respective Fourier transforms). As an application, this method provides a direct proof of the H\"older continuity in the space variable, in the case of the Riesz kernel. However, this approach cannot be used to handle the H\"older continuity in the time variable.

To illustrate the scope of our results we provide some examples of covariance functions $f$ which satisfy our conditions.
 We consider first the Riesz and Bessel kernels.  Then we focus our attention to   fractional noises of the form
\[
f(x)=|x_1|^{2H_1-2} |x_2|^{2H_2-2} |x_3|^{2H_3-2} \,,
\]
where $H_1, H_2, H_3\in (1/2, 1)$ and $\bar\kappa:=\sum_{i=1}^3 H_i-2$. We show
(see Theorem 6.1) that,  under suitable assumptions on the initial conditions,
if $\kappa_i\in (0, \min(H_i-1/2, \bar\kappa))$ and $\kappa_0=\min(\kappa_1\,, \kappa_2\,, \kappa_3)$,  then    for any bounded rectangle $I\subset \mathbb{R}^3$, there is a finite random variable  $K$, depending on the $\kappa_i$'s, such that for all $s,t  \in [0,T]$ and for all $x,y\in I$
\begin{equation*}
|u(t,x)-u(s,y)|\leq K_I
(|x_1-y_1|^{\kappa_1}+|x_2-y_2|^{\kappa_2}+|x_3-y_3|^{\kappa_3}+|s-t|^{\kappa_0}).
\end{equation*}
To see if the H\"older exponents $\kappa_i$'s are    optimal or not, we investigate
a simple linear stochastic wave equation with additive noise. That means, we consider the equation
 \eref{our SPDE}  with $v_0=\bar v_0=0$, $b=0$ and $\sigma=1$. In this situation, we prove
(see Theorem 6.2 and a Kolomogorov lemma)  that for any bounded rectangle $I\subset \mathbb{R}^3$ and for any  $\kappa\in (0, \bar \kappa)$,  there is a random variable $K_{\kappa, I}$ such that for alll $t, s \in [0,T]$ and for all $x,y\in  I$
\begin{equation*}
|u(t,x)-u(s,y)|\leq K_{\kappa,I}
(|x_1-y_1|^{\kappa}+|x_2-y_2|^{\kappa}+|x_3-y_3|^{\kappa}+|s-t|^{\kappa}).
\end{equation*}
On the other hand, we obtain in Theorem 6.2 a lower bound on the variance of the increments of the process $u$ which shows that the exponent $\bar\kappa$ is optimal.
Notice that in nonlinear case (see Theorem 6.1),   we need the extra conditions  $\kappa_i<H_i-1/2$
for $i=1, 2, 3$. Also, this extra condition is not necessary if $H_i+H_j \le 3/2$ for any $i \not=j$ (for instance, if $ H_1=H_2=H_3=H\le 3/4$), and in this case $\kappa_i$ coincides with the optimal constant  $\bar \kappa$.
It is interesting to know if the additional conditions
$\kappa_i<H_i-1/2$ are  due to the nonlinearity or due to the limitation of our technique.

This paper is organized as follows. Section 2 contains some preliminary material about the  noise process in Equation \eref{our SPDE}.   We state our basic assumptions on the covariance function $f$ and prove  a general  Burkholder inequality.  We also give the  definition of the mild solution   and state the existence and uniqueness theorem of the solution to  Equation \eref{our SPDE}. Section 3 contains  two main  results  on the H\"older continuity in the space variables. One is based on the structure of the covariance function $f$ itself and the other one uses the Fourier transform of $f$. In Section 4 we prove a criterion for  the H\"older continuity in  the time variable.  Section 5 presents  some examples of the covariance function $f$ which satisfy the conditions given in our main theorems.   In the first example,   $f$ is  the convolution of a Schwartz function with a  Riesz kernel. In the second example,   $f$ is  the Riesz kernel, which is the case studied in \cite{dalang4}. In the third example,   $f$ is the Bessel kernel.
Section 6 deals with the case when the noise process is the formal derivative of a fractional Brownian field. The optimality of the   H\"older exponents
is  discussed in this section. Section 7 contains some lemmas which are used in the paper.

\section{Preliminaries}\label{preliminaries}
Consider a  non-negative and non-negative definite function $f$
which is a tempered distribution on $\RR^3$ (so $f$ is locally
integrable). We know that in this case $f$ is the Fourier transform
of a non-negative tempered measure $\mu$ on $\RR^3$ (called the
spectral measure of $f$). That is, for all $\varphi$ belonging to
the space $\mathcal{S}(\RR^3)$ of rapidly decreasing $C^{\infty}$
functions
\begin{equation}\label{def of spectral measure mu}
\int_{\RR^3}f(x)\varphi(x)dx=\int_{\RR^3}\mathcal{F}\varphi(\xi)\mu(d\xi),
\end{equation}
and there is an integer $m\geq 1$ such that
\begin{equation}
\int_{\RR^3}(1+|\xi|^2)^{-m}\mu(d\xi)<\infty\,,
\end{equation}
where we have denoted by $\mathcal{F}\varphi$ the Fourier transform of
$\varphi\in \mathcal{S}(\RR^3)$, given by
\begin{equation*}
\mathcal{F}\varphi(\xi)=\int_{\RR^3}\varphi(x)e^{- i\xi\cdot x}dx.
\end{equation*}

Let $G(t)$ be the fundamental solution of the 3-dimensional wave
equation $\displaystyle \frac{\partial^2u}{\partial t^2}=\Delta u$.  That is
\begin{equation} \label{fundamental sol}
G(t)=\frac{1}{4\pi t}\sigma_t
\end{equation}
for any $t>0$, where $\sigma_t$ denotes the uniform surface measure
(with total mass $4\pi t^2$) on the sphere of radius $t>0$. Sometimes it is
more convenient for us to use the Fourier transform of $G$ given by
\begin{equation}\label{FourierTransofG}
\mathcal{F}G(t)(\xi)=\frac{\sin(t|\xi|)}{|\xi|}\,, \quad t>0.
\end{equation}
Our basic assumption on $f$ is
\begin{equation} \label{H1}
\int_{|x|\leq 1}\frac{f(x)}{|x|}dx<\infty.
\end{equation}
We know (see, for instance, \cite{dalang1, Nualart Quer-Sardanyons})
that this is equivalent to
\begin{equation} \label{H2}
\int_{\RR^3}\frac{\mu(d\xi)}{1+|\xi|^2}<\infty.
\end{equation}
Notice that since we are in $\RR^3$,   condition \eref{H1} is satisfied if
there is a  $\kappa<2$ such that in a neighborhood of $ 0$, $f(x)\le C |x|^{-\kappa}$.

 Fix a time interval $[0,T]$. Let
$C_0^{\infty}([0,T] \times \RR^{3})$ be the space of infinitely
differentiable functions with compact support on $[0,T] \times
\R^3$. Consider a zero mean Gaussian family of random variables
$W=\{W(\varphi), \varphi \in C_0^{\infty}([0,T] \times\RR^{3})\}$,
defined in a complete probability space $(\Omega, \mathcal{F}, P)$,
with covariance
\begin{equation}\label{Prelim covariance}
\EE(W(\varphi)W(\psi))=\int_0^{T}\int_{\RR^3}\int_{\RR^3}\varphi(t,x)f(x-y)\psi(t,y)dxdydt.
\end{equation}

Walsh's classical theory of stochastic integration developed in
\cite{Walsh} can not be applied directly to the mild formulation of
Equation \eref{our SPDE}.  We shall use the stochastic integral defined
in   Section
2.3 of \cite{dalang5}.   We briefly summarize the construction and properties
of this integral.

Let $U$ be the completion of
$C_0^{\infty}(\mathbb{R}^3)$ endowed with the inner product
\begin{equation} \label{def U}
\langle\varphi,\psi\rangle_U=\int_{\RR^3}dx\int_{\RR^3}dy\varphi(x)f(x-y)\psi(y)=\int_{\RR^3}|\mathcal{F}(\varphi)(\xi)|^2\mu(d\xi),
\end{equation}
$\varphi,\psi \in C_0^{\infty}(\RR^3)$. Set $U_T=L^2([0,T];U)$.

The Gaussian family $W$ can be extended to the space $U_T$. We will also denote by
$W(g)$ the Gaussian random variable associated with an element $g
\in U_T$. Set $W_t(h)=W({\bf 1}_{[0,t]}h)$ for any $t\in [0,T]$ and
$h \in U$. Then $W=\{ W_t, t\in [0,T]\}$ is a cylindrical Wiener
process in the Hilbert space $U$. That is, for any $h \in U$,
$\{W_t(h), t\in [0,T]\}$ is a Brownian motion with variance
$t\| h\| ^2_U$, and
\begin{equation*}
\EE(W_t(h)W_s(g))=(s\wedge t)\langle h,g\rangle_U .
\end{equation*}
Let $\mathcal{F}_t$ be the $\sigma$-field generated by the random
variables $\{W_s(h), h\in U, 0\leq s \leq t\}$ and the $P$-null
sets. We define the predictable $\sigma$-field as the $\sigma$-field
in $\Omega\times [0,T]$ generated by the sets $\{(s,t]\times A,
0\leq s< t\leq T, A\in \mathcal{F}_s\}$. Then we can define the
stochastic integral of  a $U$-valued square integrable predictable
process $g \in L^2(\Omega\times[0,T];U)$ with respect to the
cylindrical Wiener process $W$, denoted by
\begin{equation*}
g\cdot W=\int_0^T\int_{\RR^3}g(t,x)W(dt,dx),
\end{equation*}
and we have the isometry property
\begin{equation}\label{isometry property}
\EE|g\cdot W|^2=\EE\int_0^T || g(t)||^2_U dt.
\end{equation}

The following  lemma   provides a sufficient
condition for a measure of the form $\varphi(x) G(t,dx)$ to be in
the space $U$.
\begin{lemma} \label{l.2.1}  Consider a Borel measurable function
$\varphi:\RR^3 \rightarrow \RR$, such that for some $t>0$,
\begin{equation}\label{|phi| int}
\int_{\RR^3} \int_{\RR^3} | \varphi(x) \varphi(y)| G(t,dx) G(t,dy)
f(x-y) <\infty.
\end{equation}
Then, $\varphi  G(t)$ belongs to $U$ and
\begin{equation}\label{phiG in U}
\| \varphi G(t) \|^2_U= \int_{\RR^3} \int_{\RR^3}   \varphi(x)
\varphi(y) G(t,dx) G(t,dy) f(x-y).
\end{equation}
Furthermore, when $\varphi$ is bounded,
\begin{equation}\label{F varphi G id}
\int_{\RR^3} \int_{\RR^3}   \varphi(x)
\varphi(y) G(t,dx) G(t,dy) f(x-y)=\int_{\RR^3}\left|\mathcal{F}\left(\varphi G(t)\right)(\xi)\right|^2\mu(d\xi).
\end{equation}
\end{lemma}

\begin{proof}
Suppose first  that $\varphi$ is bounded.
Then by   Lemma \ref{varphi G Fourier id},   the equality (\ref{F varphi G id}) holds   and  $\int_{\RR^3}|\mathcal{F}(\varphi G(t))(\xi)|^2\mu(d\xi) < \infty$.
Let $\psi$ be a nonnegative $C^{\infty}$
function on $\RR^3$ supported in the unit ball such that
$\int_{\mathbb{R}^3} \psi(x)dx=1$. Define $\psi_n(x)=n^3\psi(nx)$,
so $\left(\psi_n*(\varphi G(t))\right)(x)$ is in $C_0^{\infty}(\RR^3)$,
and we have
\begin{eqnarray*}
&&\int_{\RR^3}|\mathcal{F}\left(\psi_n*(\varphi G(t))-\varphi
G(t)\right)|^2\mu(d\xi)\\
&=&\int_{\RR^3} |(\mathcal{F}\psi_n)(\xi)-1|^2|\mathcal{F}(\varphi
G(t))(\xi)|^2\mu(d\xi) \to \infty
\end{eqnarray*}
as $n \to \infty$, by  the dominated convergence theorem. This
implies that  $\varphi G(t)$ is in $U$, and   \eref{phiG in U} holds.

In the general case, we consider  the sequence of functions
$\varphi_k(x)= \varphi(x) \mathbf{1} _{\{ |\varphi| \le k\}}$. Then
$\varphi_k(x)G(t,dx)$ belongs to $U$, and
\begin{eqnarray*}
&&\|\varphi_k(x)G(t,dx)-\varphi(x)G(t,dx)\|^2_U\\
&\leq&\int_{\RR^3}\int_{\RR^3}|\varphi_k(x)-\varphi(x)||\varphi_k(y)-\varphi(y)|G(t,dx)G(t,dy)f(x-y),
\end{eqnarray*}
which clearly goes to $0$ as $k$ goes to infinity, by dominated
convergence theorem.
\end{proof}

 For any $x\in \mathbb{R}^3$ we denote by $G(t,x-dy)$ the shifted measure $A\rightarrow G(t,x-A)$.
 Clearly Lemma \ref{l.2.1}  holds if we replace the kernel $G(t,dy)$ by the shifted kernel $G(t,x-dy)$.
 Applying  Lemma \ref{l.2.1}, we immediately get the following Burkholder inequality.

\begin{lemma}\label{Burkholder}
Let $Z=\{Z(t,x), (t,x)\in [0,T]\times \RR^3\}$ be a predictable
process such that for some $p \geq 2$ and $x\in \RR^3$,
\begin{equation*}
\EE\left(\int_0^t
\int_{\RR^3}\int_{\RR^3}|Z(s,x-y)Z(s,x-z)|G(s,dy)G(s,dz)f(y-z)ds\right)^{\frac{p}{2}}<\infty.
\end{equation*}
Then    the measure-valued predictable process $Z(s,y)G(s,x-dy)$  belongs $L^2(\Omega\times[0,T];U)$ and   there exists a positive constant $C_p$, depending only on $p$,
such that
\begin{eqnarray*}
&&\EE\left|\int_0^t\int_{\RR^3}Z(s,y)G(s,x-dy)W(ds,dy)\right|^p\\
&\leq& C_p \EE\left(\int_0^t
\int_{\RR^3}\int_{\RR^3}Z(s,x-y)Z(s,x-z)G(s,dy)G(s,dz)f(y-z)ds\right)^{\frac{p}{2}}.
\end{eqnarray*}
\end{lemma}

If we have
\begin{equation}
\sup_{(t,x)\in [0,T]\times \RR^3}\EE|Z(t,x)|^p< \infty\,,
\end{equation}
then an application of H\"older inequality yields
\begin{eqnarray*}
&&\EE\left|\int_0^t\int_{\RR^3}Z(s,y)G(s,x-dy)W(ds,dy)\right|^p\\
&\leq& C_p\int_0^t ds \left(\sup_{x \in \RR^3}\EE
|Z(s,x)|^p\right)\left(\int_{\RR^3}\int_{\RR^3}f(y-z)G(s,dy)G(s,dz)\right)^{\frac{p}{2}}.
\end{eqnarray*}

By Lemma \ref{conv of G G lemma}, the above inequality can also be
written as
\begin{eqnarray*}
&&\EE\left|\int_0^t\int_{\RR^3}Z(s,y)G(s,x-dy)W(ds,dy)\right|^p\\
&\leq& C_p\int_0^t ds \left(\sup_{x \in \RR^3}\EE
|Z(s,x)|^p\right)\left(\int_{|x|\leq
2s}\frac{f(x)}{|x|}dx\right)^{\frac{p}{2}}.
\end{eqnarray*}

Using the above notion of stochastic integral one can introduce the
following definition:

\begin{definition}\label{def of mild solution}
A real-valued predictable stochastic process $u=\{u(t, x),0 \leq t
\leq T\,, x\in \RR^3\}$ is a mild random field solution of \eref{our
SPDE} if for all $t\in (0, T]$, $x\in \RR^3$,
\begin{eqnarray*}
 u(t, x)
&=&\frac{d}{dt} \left(G(t)*v_0\right)(x)+
\left(G(t)*\bar{v}_0\right)(x) \\
&&+\int_0^t\int_{\RR^3}G(t-s,x-dy)\si(s,y,u(s,y)) W(ds,dy)\\
&&+\int_0^tG(t-s)*\left(b(s,\cdot,u(s,\cdot))\right)(x)ds \,\ \ \ \
a.s.
\end{eqnarray*}
\end{definition}

Consider the following condition.

\noindent (\textbf{H}) The coefficients $\sigma $ and $b$ satisfy
\begin{eqnarray*}
|\sigma(t,x,u)-\sigma(t,y,v)| &\leq& C (|x-y|+|u-v|)\\
|\sigma(s,x,u)|&\leq& C(1+|u|)
\end{eqnarray*}
and
\begin{eqnarray*}
|b(t,x,u)-b(t,y,v)| &\leq& C (|x-y|+|u-v|)\\
|b(s,x,u)|&\leq& C(1+|u|)
\end{eqnarray*}
for any $x,y \in \R^3$, $s,t \in [0,T]$ and $u,v \in \RR$.

Then one can prove the existence and uniqueness of the solution to
\eref{our SPDE} in exactly the same way as in \cite[Theorem
4.3]{dalang5}.
\begin{theorem}\label{existence and uniqueness THM}
Suppose the condition \eref{H1} holds, and $\sigma $, $b$ satisfy  the
condition (\textbf{H}). Let $v_0 \in C^1(\RR^3)$ such that $v_0$ and
$\nabla v_0$ are bounded and $\bar{v}_0$ is bounded and continuous.
Then there exists a unique mild random field solution $u$ to
\eref{our SPDE} such that for all $p\geq 1$,
\begin{equation}
\sup_{(t,x) \in [0,T]\times\mathbb{R}^3}\EE|u(t,x)|^p< \infty.
\label{moments estimate}
\end{equation}
\end{theorem}

Along the paper, $C$ will denote a generic constant which may change
from line to line.

\section{H\"older continuity in the space variable}
In this section we will prove the following two theorems which are the
main results on  the H\"older continuity of the solution of Equation
(\ref{our SPDE}) in the space variable.
\begin{theorem}\label{space Holder Thm no Fourier}
Let $u$ be the solution to Equation \eref{our SPDE}.  Assume
the following conditions.
\begin{description}
\item{(a)}
The coefficients  $\sigma $ and $b$ satisfy  condition (\textbf{H}).
\item{(b)}
$v_0 \in C^2(\RR^3)$, $v_0$, $\nabla v_0$ and $\bar{v}_0$ are
bounded and
 $\Delta v_0$ and $\bar{v}_0$ are H\"older
  continuous  of orders $\gamma_{1}$ and $\gamma_{2}$
  respectively, $\gamma_{1}, \gamma_{2} \in (0,1]$.
\item{(c)} The function $f$ satisfies  condition (\ref{H1}) and for some
$\gamma \in (0,1]$ and $\gamma^{\prime} \in (0,2]$ we have for all $w\in \RR^3$
\begin{equation}\label{sp cond 1}
\int_{|z| \le 2T} \frac {|f(z+w)-f(z)|}{|z|} dz \leq  C |w|
^{\gamma}
\end{equation}
and
\begin{equation}\label{sp cond 2}
\int_{|z| \le 2T}   \frac {|f(z+w)+f(z-w)-2f(z)|}{|z|}dz \leq C
|w|^{\gamma^{\prime}}.
\end{equation}
\end{description}
Set $\kappa_1= \min(\gamma_1,\gamma_2,\gamma,\frac{\gamma^{\prime}}{2})$. Then for any $q\ge 2$,   there exists a constant $C$ such that
\begin{equation*}
\sup_{t \in [0,T]}\EE|u(t,x)-u(t,y)|^q\leq  C |x-y|^{q\kappa_1}
\end{equation*}
for any $x,y \in \mathbb{R}^3$.
\end{theorem}

\begin{proof}
   Let us assume that   $|x-y|\leq 1$ and set $x-y=w$. Fix $q \geq 2$. Then we have
\begin{eqnarray*}
\EE\left|u(t,x)-u(t,y)\right|^q
&\leq&C\Bigg\{\EE\Big|\int_0^t\int_{\mathbb{R}^3}G(t-s,x-dz)\sigma\left(s,z,u(s,z)\right)W(ds,dz)\\
&&\ \ -\int_0^t\int_{\mathbb{R}^3}G(t-s,y-dz)\sigma\left(s,z,u(s,z)\right)W(ds,dz)\Big|^q\\
&&+\EE\Big|\int_0^tG(t-s)*b\left(s,\cdot,u(s,\cdot)\right)(x)ds \\
&& \ \ -\int_0^tG(t-s)*b\left(s,\cdot,u(s,\cdot)\right)(y)ds\Big|^q\\
&&+\left|\frac{d}{dt}\left(G(t)*v_0\right)(x)-\frac{d}{dt}\left(G(t)*v_0\right)(y)\right|^q\\
&&+\left|\left(G(t)*\bar{v}_0\right)(x)-\left(G(t)*\bar{v}_0\right)(y)\right|^q\Bigg\}\\
&:=&C(I_1+I_2+I_3+I_4).
\end{eqnarray*}

For $I_4$, since $\bar{v}_0$ is H\"older continuous with exponent
$\gamma_{2}$ we get
\begin{eqnarray}
\left |G(t)*\bar{v}_0(x)-G(t)*\bar{v}_0(y)\right |^q
&\leq& \left|\int_{\mathbb{R}^3}G(t,dz)\left|\bar{v}_0(x-z)-\bar{v}_0(y-z)\right|\right|^q \nonumber\\
&\leq & C|w|^{\gamma_{2}q}\left|\int_{\mathbb{R}^3}G(t,dz)\right|^q
\le C|w|^{\gamma_{2}q}. \label{I_4 estimate}
\end{eqnarray}

For $I_3$,  we  use the identity (see, for instance, \cite{Sogge})
\begin{equation}
\frac{d}{dt}G(t)*v_0(x)=\frac{1}{t}(v_0*G(t))(x)+\frac{1}{4\pi}\int_{|y|<1}(\Delta
v_0)(x+ty)dy.
\end{equation}
 Then,  since $\Delta v_0$ is H\"older continuous with exponent
$\gamma_{1}$, we get
\begin{eqnarray}
\left|\frac{d}{dt}\left(G(t)*v_0\right)(x)-\frac{d}{dt}\left(G(t)*v_0\right)(y)\right|^q
&\leq& \frac{C}{t^q}\left|\int_{\mathbb{R}^3}G(t,dz)\left(v_0(x-z)-v_0(y-z)\right)\right|^q \nonumber\\
&& + C \left  |\int_{|z|<1}\left(\Delta v_0(x+tz)-\Delta
v_0(y+tz)\right)dz\right|^q\nonumber \\
&\le&  C|w|^{\gamma_{1}q}.\label{I3estimate}
\end{eqnarray}

For $I_2$, we use the Lipschitz condition on $b$ and  H\"older's
inequality to get
\begin{eqnarray}
I_2 &=& \EE\Big|\int_0^t\int_{\mathbb{R}^3}G(t-s,dz)b\left(s,x-z,u(s,x-z)\right)ds \nonumber \\
\nonumber&&-\int_0^t\int_{\mathbb{R}^3}G(t-s,dz)b\left(s,y-z,u(s,y-z)\right)ds\Big|^q\\
\nonumber&\leq& C\EE\left(\int_0^t\int_{\mathbb{R}^3}G(t-s,dz)\left(|w|+\left|u(s,x-z)-u(s,y-z)\right|\right)ds\right)^q\\
\nonumber&\leq&
C\EE\left(\int_0^t\int_{\mathbb{R}^3}G(t-s,dz)ds\right)^{q-1}
\Bigg(\int_0^t\int_{\mathbb{R}^3}G(t-s,dz)|w|^q ds\\
\nonumber&&+\int_0^t\int_{\mathbb{R}^3}G(t-s,dz)
\left|u(s,x-z)-u(s,y-z)\right|^q
ds\Bigg)\\
&\leq & C|w|^q+C\int_0^t
ds\sup_{z_1-z_2=w}\EE\left|u(s,z_1)-u(s,z_2)\right|^q.
\label{I2estimate}
\end{eqnarray}

For $I_1$, we apply the Burkholder inequality in
Lemma \ref{Burkholder} to get
\begin{eqnarray*}
I_{1}  &\leq&C
\EE\Big|\int_0^t\int_{\mathbb{R}^3}\int_{\mathbb{R}^3}\sigma
\left(s,\xi,u(s,\xi)\right)  f(\xi-\eta)\sigma\left(s,\eta,u(s,\eta)\right ) \\
&&\times
\left( G(t-s,x-d\xi)- G(t-s,y-d\xi) \right)\left(G(t-s,x-d\eta)-G(t-s,y-d\eta)\right)
 ds\Big|^{\frac{q}{2}}\\
&:=&C \EE |Q|^{\frac{q}{2}}.
\end{eqnarray*}
 The main idea to estimate the above quantity is to  transfer the increments of $G$ to
 increments of $f$ and $\sigma$. We introduce the following
notation
\begin{eqnarray}\label{Sigma x}
\Sigma_{x}(s,\xi)&=&\sigma\left(s,x-\xi,u(s,x-\xi)\right) \\
\label{Sigma x,y}
\Sigma_{x,y}(s,\xi)&=&\sigma\left(s,x-\xi,u(s,x-\xi)\right)-\sigma\left(s,y-\xi,u(s,y-\xi)\right).
\end{eqnarray}
Define
\begin{eqnarray}\label{h_1}
h_1&=&f(\eta-\xi)\Sigma_{x,y}(s,\xi)\Sigma_{x,y}(s,\eta),\\
\label{h_2}
h_2&=&\left(f(\eta-\xi+w)-f(\eta-\xi)\right)\Sigma_{x}(s,\xi)\Sigma_{x,y}(s,\eta), \\
\label{h_3}
h_3&=&\left(f(\eta-\xi-w)-f(\eta-\xi)\right)\Sigma_{x}(s,\eta)\Sigma_{x,y}(s,\xi),\\
\label{h_4}
h_4&=&\left(2f(\eta-\xi)-f(\eta-\xi+w)-f(\eta-\xi-w)\right)\Sigma_{x}(s,\xi)\Sigma_{x}(s,\eta).
\end{eqnarray}
and
\begin{equation}\label{Qi}
Q_{i} =\int_0^t\int_{\mathbb{R}^3}\int_{\mathbb{R}^3}
G(t-s,d\xi)G(t-s,d\eta)h_i ds, \quad i=1,2,3,4.
\end{equation}
Then by direct calculation, we can verify that  $Q=\sum_{i=1}^4
Q_{i}$. To estimate $\EE |Q|^{\frac{q}{2}}$, we need to estimate
$\EE|Q_{i}  |^{\frac q2} $ for $i=1, \dots, 4$. For
$\EE|Q_{1}|^{\frac{q}{2}}$, by the assumptions on $\sigma$, using
 H\"older's inequality and Lemma \ref{conv of G G lemma} we have
\begin{eqnarray*}
\EE\left|Q_{1}\right|^{\frac{q}{2}} &=&\EE\Big|\int_0^t
ds\int_{\mathbb{R}^3}\int_{\mathbb{R}^3}
G(t-s,d\xi)G(t-s,d\eta)f(\eta-\xi)\\
&&\times\left(\sigma\left(s,x-\xi,u(s,x-\xi)\right)-\sigma\left(s,y-\xi,u(s,y-\xi)\right)\right)\\
&&\times\left(\sigma\left(s,x-\eta,u(s,x-\eta)\right)-
\sigma\left(s,y-\eta,u(s,y-\eta)\right)\right)\Big|^{\frac{q}{2}}\\
&\leq& C
\EE\Big|\int_0^t ds\int_{\mathbb{R}^3}\int_{\mathbb{R}^3} G(t-s,d\xi)G(t-s,d\eta)f(\eta-\xi)\\
&&\times\left(|w|+|u(s,x-\xi)-u(s,y-\xi)|\right)\\
&&\times\left(|w|+|u(s,x-\eta)-u(s,y-\eta)|\right)\Big|^{\frac{q}{2}}\\
&\leq&C \int_0^t ds \left( \int_{\mathbb{R}^3}\int_{\mathbb{R}^3}
 G(t-s,d\xi)G(t-s,d\eta)f(\eta-\xi)\right)^{\frac{q}{2}}\\
&&\times \left(|w|^{q} +\sup_{z_1-z_2=w}\EE|u(s,z_1)-u(s,z_2)|^{q
}\right)\\
&=&C \int_0^t ds \left( \int_{\mathbb{R}^3}
 G(t-s)*G(t-s)(dz)f(z)\right)^{\frac{q}{2}}\\
&&\times \left(|w|^{q} +\sup_{z_1-z_2=w}\EE|u(s,z_1)-u(s,z_2)|^{q
}\right)\\
&\leq&C\int_0^t ds \left( \int_{|z|\leq 2T}
 \frac{f(z)}{|z|}dz\right)^{\frac{q}{2}} \left(|w|^{q} +\sup_{z_1-z_2=w}\EE|u(s,z_1)-u(s,z_2)|^{q
}\right).
\end{eqnarray*}
By the condition (\ref{H1}), we get
\begin{equation}\label{Q1}
\EE|Q_1|^{\frac{q}{2}}\leq C|w|^q+C\int_0^t ds
\sup_{z_1-z_2=w}\EE|u(s,z_1)-u(s,z_2)|^{q}.
\end{equation}
For $\EE|Q_{2}|^{\frac{q}{2}}$,  we write  $f(\eta-\xi
+w)-f(\eta-\xi)=\Phi_1(\eta-\xi,w)$ and using the inequality
 $ab\leq \frac{a^2+b^2}{2}$ we obtain
\begin{eqnarray*}
\EE|Q_2|^{\frac{q}{2}}&\leq& C \EE \left(\int_0^t
\int_{\RR^3}\int_{\RR^3}   |\Phi_1(\eta-\xi,w)| |\Sigma_x(s,\xi)||\Sigma_{x,y}(s,\eta)|G(t-s,d\xi)G(t-s,d\eta)ds\right)^{\frac{q}{2}}\\
&\leq&C  \EE\left(\int_0^t\int_{\RR^3}\int_{\RR^3} |w|^\gamma|\Phi_1(\eta-\xi,w)||\Sigma_x(s,\xi)|^2G(t-s,d\xi)G(t-s,d\eta)ds\right)^{\frac{q}{2}}\\
&&+C\EE\left(\int_0^t\int_{\RR^3}\int_{\RR^3} \frac {|\Phi_1(\eta-\xi,w)|} {|w|^\gamma}|\Sigma_{x,y}(s,\eta)|^2G(t-s,d\xi)G(t-s,d\eta)ds\right)^{\frac{q}{2}}\\
&:=&C(Q_{2,1}+Q_{2,2}).
\end{eqnarray*}
Applying the condition \eref{sp cond 1}, Lemma \ref{conv of G G lemma}
and  H\"older's inequality yields
\begin{eqnarray*}
Q_{2,1}&\leq&|w| ^{\frac {q\gamma}2}\left(\int_0^t\int_{\RR^3}\int_{\RR^3} | \Phi_1(\eta-\xi, w)|G(t-s,d\xi)G(t-s,d\eta)ds\right)^{\frac{q}{2}}\sup_{s,x}\EE|\Sigma_{x}(s,\xi)|^q\\
&\leq&C|w|^{\frac{q\gamma}2}\left(\int_0^t\int_{|z|\leq 2T} \frac {
|f(z+w)-f(z)|}{ |z|} dzds\right)^{\frac{q}{2}}\leq C|w|^{q\gamma}.
\end{eqnarray*}
For the second term we obtain
\begin{eqnarray*}
Q_{2,2}&\leq&C |w| ^{-\frac {q\gamma}2} \EE\Big(\int_0^t
\int_{\RR^3}\int_{\RR^3}   | \Phi_1(\eta-\xi,w)|\left(|w|^2+\left|u(s,x-\xi)-u(s,y-\xi)\right|^2\right)\\
&&\times G(t-s,d\xi)G(t-s,d\eta)ds\Big)^{\frac{q}{2}}\\
&\leq&C|w|^{q-\frac {q\gamma}2} \left(\int_0^t
\int_{|z|\leq 2T}    \frac {| f(z+w)-f(z)|}{|z|}  dzds\right)^{\frac{q}{2}}\\
&&+C|w|^{-\frac {q\gamma}2}  \left(\int_{|z|\leq2T}    \frac {| f(z+w)-f(z)|}{|z|}dz\right)^{\frac{q}{2}}\int_0^t\sup_{z_2-z_1=w}\EE|u(s,z_1)-u(s,z_2)|^qds\\
&\leq&C|w|^q+C\int_0^t\sup_{z_2-z_1=w}\EE|u(s,z_1)-u(s,z_2)|^qds.
\end{eqnarray*}
So we conclude that
\begin{equation}\label{Q2}
\EE|Q_2|^{\frac{q}{2}}\leq
C|w|^{q\gamma}+C\int_0^t\sup_{z_2-z_1=w}\EE|u(s,z_1)-u(s,z_2)|^qds.
\end{equation}
The term $\EE|Q_3|^{\frac{q}{2}}$ can be treated in the same way and we have
\begin{equation}\label{Q3}
\EE|Q_3|^{\frac{q}{2}}\leq C
|w|^{q\gamma}+C\int_0^t\sup_{z_2-z_1=w}\EE|u(s,z_1)-u(s,z_2)|^qds.
\end{equation}
For $\EE|Q_4|^{\frac{q}{2}}$,  we set $\Phi_2(\eta-\xi,w)=f(\eta-\xi+w) +f(\eta-\xi-w) -2f(\eta-\xi)$, and using the  assumption
on $\si$,  condition \eref{sp
cond 2},  H\"older's inequality and the moments
estimate \eref{moments estimate}, we have
\begin{eqnarray*}
\EE|Q_4|^{\frac{q}{2}}&=&\EE\left(\int_0^t
\int_{\RR^3}\int_{\RR^3} |\Phi_2(\eta- \xi,w)| |\Sigma_x(s,\xi)\Sigma_x(s,\eta)|G(t-s,d\xi)G(t-s,d\eta)ds\right)^{\frac{q}{2}}\\
&\leq& \left(\int_0^t\int_{\RR^3}\int_{\RR^3}  |\Phi_2(\eta- \xi,w)| G(t-s,d\xi)G(t-s,d\eta)ds\right)^{\frac{q}{2}}\sup_{s,\xi,\eta}\EE\left(|\Si_x(s,\xi)||\Sigma_x(s,\eta)|\right)^{\frac{q}{2}}\\
&\leq&C \left(\int_0^t \int_{|z|\leq 2T}     \frac {|f(z+w) +f(z-w)
-2f(z)|} {|z|}   dzds\right)^{\frac{q}{2}}\leq C
|w|^{\frac{q\gamma^{\prime}}{2}}.
\end{eqnarray*}
Combining the above expression   with \eref{Q1}, \eref{Q2} and \eref{Q3},
we can write
\begin{equation}
I_1\leq C(|w|^{\gamma
q}+|w|^{\frac{\gamma^{\prime}q}{2}})+C\int_0^t
\sup_{z_2-z_1=w}\EE|u(s,z_1)-u(s,z_2)|^qds.
\end{equation}
The estimates for $I_i$, $i=1,2,3,4$, lead to
\begin{eqnarray*}
&&\sup_{z_1-z_2=w}\EE|u(t,z_1)-u(t,z_2)|^q\\
&\leq&C|w|^{q\min(\gamma_1,\gamma_2,\gamma,\frac{\gamma^{\prime}}{2})}+C\int_0^t
ds \sup_{z_1-z_2=w}|u(s,z_1)-u(s,z_2)|^q.
\end{eqnarray*}
An application of Gronwall's lemma yields
\begin{equation}
\EE|u(t,x)-u(t,y)|^q\leq C
|x-y|^{q\min(\gamma_1,\gamma_2,\gamma,\frac{\gamma^{\prime}}{2})}
\end{equation}
for any $x$ and $y$ in $\RR^3$ such that $|x-y|\leq 1$, which completes the proof of the theorem.
\end{proof}

Next we give a theorem which establishes the H\"older continuity
in the space variable using the Fourier transform.

\begin{theorem}\label{space Holder Thm}
Let $u$ be the solution to Equation \eref{our SPDE}.  Assume
conditions (a) and (b) in Theorem \ref{space Holder Thm no Fourier}. Suppose the following condition:
\begin{description}
\item{(a)}  For some $\gamma$ such that $0<\gamma \le 1$, the Fourier transform of the tempered measure $|\zeta|^{2\gamma} \mu (d\zeta)$ is a nonnegative locally integrable function and
\begin{equation} \label{eu4}
\int_{\RR^3} \frac  {\mu (d\zeta)}{1+|\zeta|^{2-2\gamma}} <\infty.
\end{equation}
\end{description}
Set $\kappa_1= \min(\gamma_1,\gamma_2,\gamma)$. Then for any $q\ge 2$,   there exists a constant $C$ such that
\begin{equation*}
\sup_{t \in [0,T]}\EE|u(t,x)-u(t,y)|^q\leq  C |x-y|^{q\kappa_1}
\end{equation*}
for any $x,y \in \mathbb{R}^3$.
\end{theorem}
\begin{proof}
Assume that $|x-y|\leq 1$ and set $x-y=w$.  Fix $q \geq 2$, as in the proof of
Theorem \ref{space Holder Thm no Fourier}, we still express
$\EE|u(t,x)-u(t,y)|^q$ as $C(I_1+I_2+I_3+I_4)$, and the estimates
for $I_2$, $I_3$, $I_4$ are the same as in the proof of Theorem
\ref{space Holder Thm no Fourier}. For $I_1$, use the notation \eref{Sigma x}
-\eref{Qi}  and  we need to  estimate $\EE |Q_i|^{\frac{q}{2}}$ for $i=1, \dots , 4$.

The estimate for $\EE |Q_1|^{\frac{q}{2}}$ is the same as in the proof of Theorem \ref{space Holder Thm no Fourier}.

For $Q_2$  we  would like to apply Equation (\ref{eu3}) to $\varphi=\Sigma_{x,y}(s,\eta)$ and
$\psi=\Sigma_{x}(s,\xi)$. Because there functions are not necessarily bounded we
 we introduce the truncations
\begin{eqnarray}\label{truncation1}
\Sigma_{x}^k(s,\xi)&=& \Sigma_{x}(s,\xi)
\mathbf{1}_{\{|\Sigma_{x}(s,\xi)|\leq k\}}, \\
\Sigma_{x,y}^k(s,\eta)&=&
\Sigma_{x,y}(s,\eta)\mathbf{1}_{\{|\Sigma_{x,y}(s,\eta)|\leq k\}}\,,
\end{eqnarray}
 for any $k>0$.  Clearly, as $k$ tends to infinity,  $\Sigma_{x}^k(s,\xi)$ and
$\Sigma_{x,y}^k(s,\eta)$ converge point-wise to  $\Sigma_{x}(s,\xi)$
and $\Sigma_{x,y}(s,\eta)$, respectively.  Set
\begin{equation*}
Q_2^k=\int_0^t ds \int_{\RR^3} \int_{\RR^3}
G(t-s,d\xi)G(t-s,d\eta)\left(f(\eta-\xi+w)-f(\eta-\xi)\right)\Sigma^k_{x}(s,\xi)\Sigma^k_{x,y}(s,\eta).
\end{equation*}
Then  Equation (\ref{eu3}) yields
\[
Q_2^k=\int_0^t ds \int_{\RR^3} \overline{  \mathcal{F}\left(\Sigma_x^k(s,\cdot)G(t-s)\right)(\zeta)}
\mathcal{F}\left( \Sigma_{x,y}^k(s,\cdot)G(t-s)\right)(\zeta)(e^{-iw\cdot \zeta}-1)\mu(d\zeta)
\]
Using the estimate $| e^{-iw\cdot \zeta}-1|\leq C
|w|^{\gamma}|\zeta|^{\gamma}$ for every $0 < \gamma \leq 1$,  Cauchy-Schwartz's inequality and the inequality $\sqrt{ab} \le \frac 12 (a +b)$ for any $a,b>0$, we can write
\begin{eqnarray*}
|Q_{2}^k| &\le &\int_0^t ds
\int_{\RR^3}\left|\mathcal{F}\left(\Sigma_x^k(s,\cdot)G(t-s)\right)(\zeta)\right|\left|\mathcal{F}\left(\Sigma_{x,y}^k(s,\cdot)G(t-s)\right)(\zeta)\right|  |w|^{\gamma}|\zeta|^{\gamma}\mu(d\zeta)\\
&\leq&\int_0^t ds |w|^{\gamma}\left(\int_{\RR^3}\left|\mathcal{F}\left(\Sigma_x^k(s,\cdot)G(t-s)\right)(\zeta)\right|^2|\zeta|^{2\gamma}\mu(d\zeta)\right)^{\frac{1}{2}}\\
&&\qquad\times\left(\int_{\RR^3}\left|\mathcal{F}\left(\Sigma_{x,y}^k(s,\cdot)G(t-s)\right)(\zeta)\right|^2\mu(d\zeta)\right)^{\frac{1}{2}}\\
&\leq&\frac{1}{2}\int_0^t ds |w|^{2\gamma}\int_{\RR^3}\left|\mathcal{F}\left(\Sigma_x^k(s,\cdot)G(t-s)\right)(\zeta)\right|^2|\zeta|^{2\gamma}\mu(d\zeta)\\
&&+\frac{1}{2}\int_0^t ds \int_{\RR^3} \left |\mathcal{F}\left(\Sigma_{x,y}^k(s,\cdot)G(t-s)\right)(\zeta)\right|^2\mu(d\zeta)\\
&=&\frac{1}{2}|w|^{2\gamma}\int_0^t ds \int_{\RR^3} d\eta g(\eta) \left(\Sigma_x^k(s,\cdot)G(t-s)\right)*\left(\widetilde{\Sigma_x^k(s,\cdot)G(t-s)}\right )(\eta)\\
&&+\frac{1}{2}\int_0^t ds \int_{\RR^3} d\eta f(\eta)\left(\Sigma_{x,y}^k(s,\cdot)G(t-s)\right)*\left(\widetilde{\Sigma_{x,y}^k(s,\cdot)G(t-s)}\right)(\eta)\\
&:=& Q^k_{2,1} + Q^k_{2,2},
\end{eqnarray*}
where $g$ is the Fourier transform of the measure $|\cdot|^{2\gamma}
\mu$, which by our hypothesis is a nonnegative locally integrable
function. In the above formula,   for any measure $\nu$, $\widetilde{\nu}$
denotes the measure  $\widetilde{\nu}(A)=\nu(-A)$.
Treating $g(\eta)
G(t-s)*G(t-s)(\eta) d\eta$ as a new measure, and using
Minkowski's's inequality, we get
\begin{eqnarray*}
\EE|Q_{2,1}^k|^{\frac{q}{2}} &\leq&C|w|^{q\gamma}\int_0^t ds
\left(\int_{\RR^3}d\eta g(\eta)
G(t-s)*G(t-s)(\eta)\right)^{\frac{q}{2}}\\
&&\times\sup_{0\leq s\leq t, x,\xi,\eta \in
\RR^3}  \EE\left|\Sigma_{x}(s,\xi)\Sigma_{x}(s,\xi+\eta)\right|^{\frac{q}{2}} \\
&\leq&C|w|^{q\gamma}\int _0^t ds \left(\int_{\RR^3}|\zeta|^{2\gamma}\mu(d\zeta)\mathcal{F}\left(G(t-s)*G(t-s)\right)(\zeta)\right)^{\frac{q}{2}}\\
&\leq&C|w|^{q\gamma}\left(\int_{\RR^3}
\mu(d\zeta)\frac{|\zeta|^{2\gamma}}{1+|\zeta|^2}\right)^{\frac{q}{2}}\leq
C|w|^{q\gamma},
\end{eqnarray*}
where we have used the  moments estimate \eref{moments estimate},  Equation
\eref{FourierTransofG}, the fact that
$\left(\frac{\sin(s|\xi|)}{|\xi|}\right)^2\leq
\frac{C}{1+|\xi|^2}$, when $s\in [0,T]$  and the inequality
$|\Sigma^k_{x}(s,\xi)|\leq |\Sigma_{x}(s,\xi)|$.
Therefore,
\begin{eqnarray*}
\EE\left|Q^k_{2}\right|^{\frac{q}{2}} &\leq &C|w|^{q\gamma}
+C\EE\int_0^t ds
\left(\int_{\RR^3}d\eta f(\eta)\left(\Sigma_{x,y}^k(s,\cdot)G(t-s)\right)*\left(\widetilde{\Sigma_{x,y}^k(s,\cdot)G(t-s)}\right)(\eta)
\right)^{\frac{q}{2}}.
\end{eqnarray*}
Applying the dominated convergence theorem we can show that in the
above inequality, as $k$ goes to infinity, the left-hand side converges to
$\EE\left|Q_2\right|^{\frac{q}{2}}$ and the expectation on the right-hand
side converges to
\begin{equation*}
\EE\int_0^t ds
\left(\int_{\RR^3}f(\eta)d\eta\left(\Sigma_{x,y}(s,\cdot)G(t-s)\right)*\left(\widetilde{\Sigma_{x,y}(s,\cdot)G(t-s)}\right)(\eta)\right)^{\frac{q}{2}}.
\end{equation*}
From  the expression of $\Sigma_{x,y}(s,\xi)$ and using
Minkowski's's inequality, we have
\begin{eqnarray*}
\EE\left|Q_{2}\right|^{\frac{q}{2}}
&\leq& C|w|^{q\gamma}+ C\int_0^t ds
\left(\int_{\RR^3}d\eta
f(\eta)\left(G(t-s)*G(t-s)\right)(\eta)\right)^{\frac{q}{2}}\\
&&\times\sup_{\eta \in \RR^3}\EE \left[
\left|\sigma\left(s,x-\eta,u(s,x-\eta)\right)-\sigma\left(s,y-\eta,u(s,y-\eta)\right)\right|^q\right]\\
&\leq&C|w|^{q\gamma} +C\int_0^t ds
\left(\int_{\RR^3}\frac{d\mu(\zeta)}{1+|\zeta|^2}\right)^{\frac{q}{2}}\left(|w|^q+\sup_{z_1-z_2=w}\EE\left|u(s,z_1)-u(s,z_2)\right|^q\right)\\
&\leq&C\left( |w|^{q\gamma}+|w|^q+\int_0^t ds
\sup_{z_1-z_2=w}\EE\left|u(s,z_1)-u(s,z_2)\right|^q\right).
\end{eqnarray*}
The same estimate holds for $\EE|Q_3|^{\frac{q}{2}}$.

Consider now the term $Q_4$. We use the truncation argument as in the estimation for $\EE|Q_2|^{\frac{q}{2}}$ and we set
\begin{eqnarray*}
Q_4^k&=&\int_0^t ds \int_{\RR^3} \int_{\RR^3}
G(t-s,d\xi)G(t-s,d\eta)\left(2f(\eta-\xi)-f(\eta-\xi+w)-f(\eta-\xi-w)\right)\\
&&\quad \times \Sigma^k_{x}(s,\xi)\Sigma^k_{x}(s,\eta).
\end{eqnarray*}
Then,  Equation (\ref{eu3}) implies
\begin{eqnarray*}
 Q_{4}^k&=&\int_0^t ds \int_{\RR^3}\mu(d\zeta)\left(1-\cos(w\cdot
\zeta)\right)|\left(\mathcal{F}\left(\Sigma_{x}^k(s,\cdot)G(t-s)\right)\right)(\zeta)|^2\\
&\leq&2|w|^{2\gamma}\int_0^t ds \int_{\RR^3}
\mu(d\zeta)|\zeta|^{2\gamma}|\left(\mathcal{F}\left(\Sigma_{x}^k(s,\cdot)G(t-s)\right)\right)(\zeta)|^2\\
&=&2|w|^{2\gamma}\int_0^t ds \int_{\RR^3}d\eta g(\eta)\left(\left(\Sigma_{x}^k(s,\cdot)G(t-s)\right)*\left(\widetilde{\Sigma_{x}^k(s,\cdot)G(t-s)}\right)\right)(\eta).
\end{eqnarray*}
Then we can use the same argument as before, to conclude that
\begin{equation*}
\EE\left|Q_{4}\right|^{\frac{q}{2}}\leq C |w|^{q\gamma}.
\end{equation*}
Combining the moment estimates for $E\left|Q_{i}\right|^{\frac
q2}$, $i=1,2,3,4$, since $|w|\leq 1$ and $0<\gamma\leq 1$, we have
\begin{equation}
I_{1}\leq C|w|^{q\gamma} +C\int_0^t ds \sup_
{z_1-z_2=w}\EE\left|u(s,z_1)-u(s,z_2)\right|^q.  \label{I1nestimate}
\end{equation}
Finally,  the estimates for  $I_i$, $i=1, 2, 3,4$,  allow us to write
\begin{eqnarray*}
&&\sup_{z_1-z_2=x-y}\EE|u(t,z_1)-u(t,z_2)|^q\\
&&\leq C|x-y|^{q\min(\gamma_1,\gamma_2,\gamma)}+C\int_0^t ds
\sup_{z_1-z_2=x-y}\EE\left|u(s,z_1)-u(s,z_2)\right|^q.
\end{eqnarray*}
An application of Gronwall's lemma yields
\begin{equation}
\EE|u(t,x)-u(t,y)|^q\leq C |x-y|^{q\min(\gamma_1,\gamma_2,\gamma)}
\label{E Space Holder}
\end{equation}
for any $x$ and $y$ in $\RR^3$ such that $|x-y|\leq 1$, which completes the proof of the theorem.
\end{proof}

\medskip
Under the assumptions of Theorem 3.1  or Theorem 3.2,   applying   Kolmogorov's continuity criterion, for any fixed $t\in [0,T]$, we deduce the existence of a locally H\"older continuous version for the process $\{u(t,x), x\in \RR^3\}$ with exponent
$\kappa>0$ where $\kappa<\kappa_1$.
Namely, for any  $t\in [0,T]$ and any compact rectangle $I\subset \mathbb{R}^3$, there exists a random variable $K_{\kappa,t,I}$ such that
\begin{equation*}
|u(t,x)-u(t,y)|\leq K _{\kappa,t,I}|x-y|^{\kappa}
\end{equation*}
for   and $x,y \in I $.

\section{H\"older continuity in space and  time variables}
In this section we obtain a result on the H\"older continuity of the
solution of Equation \eref{our SPDE} in both the space and time variables.
 Let $S^2$ denote the
unit sphere in $\RR^3$ and $\sigma(d\xi)$ the uniform measure on it.
We have the following result.
\begin{theorem}\label{time holder Thm}
Let $u $ be the solution to Equation \eref{our SPDE}.  Assume
conditions (a) and (b) in Theorem \ref{space Holder Thm no Fourier}. Suppose  the following conditions hold.
\begin{description}
 \item{(1)}
For some $0 < \nu\leq 1$, $\int_{|z|\leq h}\frac{f(z)}{|z|}dz\leq C
h^{\nu}$ for any $0 <h\leq 2T$.
\item{(2)} For some $0< \kappa_1 \le 1$  and for any $q\ge 2$ and $t \in (0,T]$, we have
\[
\EE | u(t,x) -u(t,y) | ^q \le C |x-y|^{q\kappa_1}.
\]
\item{(3)} Let $\xi$ and $\eta$ be unit vectors
in $\RR^3$ and $0<h\leq 1$. We have
\begin{equation}\label{time cond 1}
 \int_0^T \int_{S^2}\int_{S^2}
  \left|f\left(s(\xi+\eta) +h(\xi+\eta) \right)-f\left(s(\xi+\eta) +h\eta\right)\right|  s \sigma(d\xi) \sigma(d\eta) ds \leq C
  h^{\rho_1},
\end{equation}
for some $\rho_1 \in (0,1]$, and
\begin{eqnarray}
 &&\int_0^T \int_{S^2}\int_{S^2} \notag
  \Big|f\left(s(\xi+\eta) +h(\xi+\eta) \right)-f\left(s(\xi+\eta) +h\xi\right) \\
 &&\qquad  - f\left(s(\xi+\eta) +h\eta\right)+f\left(s(\xi+\eta)\right) \Big|
  s^2  \sigma(d\xi) \sigma(d\eta) ds  \leq C
h^{\rho_2}, \label{time cond 2}
\end{eqnarray}
for some $\rho_2\in(0,2]$.
\end{description}
Set $\kappa_2=  \min(\gamma_1,\gamma_2,\kappa_1,\frac{\nu+1}{2},\frac{\rho_1+\kappa_1}{2},\frac{\rho_2}{2})$.
Then for any  $q\ge 2$,  there exists a constant $C$ such that
\begin{equation*}
\sup_{x \in \RR^3}\EE|u(\bar{t},x)-u(t,x)|^q\leq C|\bar{t}-t|^{\kappa_2 }
\end{equation*}
for any $t, \bar{t} \in [0,T]$.
\end{theorem}

\begin{proof}
Fix $x \in \RR^3$ and $q \in [2,\infty)$. For all $0\leq t\leq
\bar{t}\leq T$ we can write, by Definition \ref{def of mild
solution},
\begin{equation*}
\EE|u(t,x)-u(\bar{t},x)|^q\leq C\sum_{i=1}^4T_i,
\end{equation*}
where
\begin{eqnarray*}
 T_1 & =& \left|\left(\frac{d}{dt}G(t)*v_0\right)(x)-\left(\frac{d}{dt}G(\bar{t})*v_0\right)(x)\right|^q,\\
 T_2  &=& \left|(G(t)*\bar{v}_0)(x)-(G(\bar{t})*\bar{v}_0)(x)\right|^q,\\
 T_3  &=&\EE\Big|\int_0^t ds
\int_{\mathbb{R}^3}G(t-s,dy)b\left(s,x-y,u(s,x-y)\right)\\
&& \quad -\int_0^{\bar{t}} ds
\int_{\mathbb{R}^3}G(\bar{t}-s,dy)b\left(s,x-y,u(s,x-y)\right)\Big|^q,\\
 T_4 &=& \EE\Big|\int_0^t\int_{\mathbb{R}^3}G(t-s,x-dy)\sigma\left(s,y,u(s,y)\right)W(ds,dy)\\
&&
\quad -\int_0^{\bar{t}}\int_{\mathbb{R}^3}G(\bar{t}-s,x-dy)\sigma\left(s,y,u(s,y)\right)W(ds,dy)\Big|^q.
\end{eqnarray*}

Let $\gamma^{\prime}=\min(\gamma_1,\gamma_2)$. By our assumptions on
$\Delta v_0$ and $\bar{v}_0$ and by Lemma 4.9 in \cite{dalang4}, we
have
\begin{equation}\label{T_1+T_2}
 T_1 +T_2 \leq
C|t-\bar{t}|^{q\gamma^{\prime}}.
\end{equation}
Notice that Lemma 4.9 in \cite{dalang4} assumes that $x$ belongs to
a bounded set in $\RR^3$, but from the proof it is easy to see that
the constant $C$ does not depend on $x$.

The term $T_3 $ is bounded by
\begin{equation*}
T_3 \leq C(T_{3,1} +T_{3,2} ),
\end{equation*}
where
\begin{eqnarray*}
  T_{3,1} &=&\EE\left|\int_t^{\bar{t}}ds\int_{\mathbb{R}^3}
G(\bar{t}-s,dy)b\left(s,x-y,u(s,x-y)\right)\right|^q,\\
 T_{3,2} &=&\EE\left|\int_0^t ds \int
_{\mathbb{R}^3}\left(G(t-s,dy)-G(\bar{t}-s,dy)\right)b\left(s,x-y,u(s,x-y)\right)\right|^q.
\end{eqnarray*}
  H\"older's inequality,  the linear growth of $b$ and the moments
estimate \eref{moments estimate} imply
\begin{eqnarray*}
 T_{3,1} &\leq& C
\left(\int_t^{\bar{t}}ds
\int_{\mathbb{R}^3}G(\bar{t}-s,dy)\right)^{q-1}\\
&&\times\left(\int_t^{\bar{t}}ds
\int_{\mathbb{R}^3}G(\bar{t}-s,dy)\sup_{0\leq s \leq T}\sup_{x\in
\mathbb{R}^3}\left(1+\EE|u(s,x)|^q\right)\right)\\
&\leq& C (\bar{t}-t)^{ q}.
\end{eqnarray*}
For $T_{3,2} $, we split the integral into a difference of two
integrals and then we apply the change of variables
$\frac{y}{t-s}\rightarrow y$ and $\frac{y}{\bar{t}-s}\rightarrow y$,
respectively. In this way, taking into account that $G(t,dy)= t^{-2}
G(1, t^{-1} dy)$, we get
\begin{eqnarray*}
 T_{3,2}
&=&\EE\Big|\int_0^tds\int_{\mathbb{R}^3}G(1,dy)b\left(s,x-(t-s)y,u\left(s,x-(t-s)y\right)\right)(t-s)\\
&&-\int_0^tds\int_{\mathbb{R}^3}G(1,dy)b\left(s,x-(\bar{t}-s)y,u\left(s,x-(\bar{t}-s)y\right)\right)(\bar{t}-s)\Big|^q.
\end{eqnarray*}
Hence, $T_{3,2} \leq C\left (T_{3,2,1} +T_{3,2,2} \right)$, where
\[
 T_{3,2,1} =(\bar{t}-t)^q\EE\left|\int_0^t ds
\int_{\mathbb{R}^3}G(1,dy)b\left(s,x-(\bar{t}-s)y,u\left(s,x-(\bar{t}-s)y\right)\right)\right|^q
\]
and
\begin{eqnarray*}
 T_{3,2,2} &=&\EE\Big(\int_0^t ds
(t-s)\int_{\mathbb{R}^3}G(1,dy)\Big|b\left(s,x-(\bar{t}-s)y,u\left(s,x-(\bar{t}-s)y\right)\right)\\
&&
-b\left(s,x-(t-s)y,u\left(s,x-(t-s)y\right)\right)\Big|\Big)^q.
\end{eqnarray*}
By the moments estimate \eref{moments estimate} and the linear
growth of $b$, it follows that
\begin{equation*}
 T_{3,2,1} \leq C|\bar{t}-t|^q.
\end{equation*}
Moreover, by the Lipschitz property of $b$  and
  H\"older continuity assumption on  the space variable
(condition (2) in the theorem), we get
\begin{eqnarray*}
 T_{3,2,2}
&\leq&C \EE\Big(\int_0^t ds
(t-s)\int_{\mathbb{R}^3}G(1,dy)\big((\bar{t}-t)|y|  \\
&& +\left|u(s,x-(\bar{t}-s)y)-u(s,x-(t-s)y)\right|\big)\Big)^q\\
&\leq&\left(\int_0^t ds
(t-s)\int_{\mathbb{R}^3}G(1,dy)\right)^{q-1}
 \int_0^t ds (t-s)
\int_{\mathbb{R}^3}G(1,dy)\Big((\bar{t}-t)^q|y|^q \\
&&+\sup_{x \in
\RR^3}\EE\left|u\left(s,x-(\bar{t}-s)y\right)-u\left(s,x-(t-s)y\right)\right|^q\Big)\\
&\leq&C(|\bar{t}-t|^q+|\bar{t}-t|^{q\kappa})\leq C
|\bar{t}-t|^{q\kappa}.
\end{eqnarray*}
Combining the estimates for $T_{3,1} $,
$T_{3,2,1} $ and $T_{3,2,2} $ we conclude
that
\begin{equation}\label{T_3}
  T_3 \leq C|\bar{t}-t|^{q\kappa}.
\end{equation}

Next we estimate the term $T_4 $ which involves a
  stochastic integral.
Consider the decomposition
\begin{equation*}
T_{4 } \leq
C(T_{4, 1} +T_{4, 2} ),
\end{equation*}
where
 \[
 T_{4, 1 } =\EE\left|\int_t^{\bar{t}}\int_{\mathbb{R}^3}G (\bar{t}-s,x-dy)\sigma\left(s,y,u(s,y)\right)W(ds,dy)\right|^q
 \]
 and
 \[
 T_{4,2 } =\EE\left|\int_0^t\int_{\mathbb{R}^3}\left(G (\bar{t}-s,x-dy)-G(t-s,x-dy)\right)\sigma \left(s,y,u(s,y)\right) W(ds,dy)\right|^q.
\]
By the linear growth of $\sigma$ and  Burkholder's inequality
(Lemma \ref{Burkholder}), we obtain
\begin{eqnarray*}
  T_{4, 1 }
&\leq&C\EE\Big(\int_t^{\bar{t}}ds
\int_{\RR^3}\int_{\RR^3}
G (\bar{t}-s,x-dy)G (\bar{t}-s,x-dz)\\
&&\times
f(y-z)\sigma\left(s,y,u(s,y)\right)\sigma\left(s,z,u(s,z)\right)\Big)^{\frac{q}{2}}\\
&=&C\EE\Big(\int_0^{\bar{t}-t} ds
\int_{\RR^3}\int_{\RR^3}
G (s,x-dy)G (s,x-dz)\\
&&\times
f(y-z)\sigma\left(\bar{t}-s,y,u(\bar{t}-s,y)\right)\sigma\left(\bar{t}-s,z,u(\bar{t}-s,z)\right)\Big)^{\frac{q}{2}}\\
&\leq&C\EE\Big(\int_0^{\bar{t}-t} ds
\int_{\RR^3}\int_{\RR^3}
 G (s,x-dy)G (s,x-dz)\\
&&\times
f(y-z)\left(1+|u(\bar{t}-s,y)|\right)\left(1+|u(\bar{t}-s,z)|\right)\Big)^{\frac{q}{2}}.
\end{eqnarray*}
Using  H\"older's inequality, the moments estimate \eref{moments
estimate} and   condition (1), we can write
\begin{eqnarray}
 T_{4, 1}
\nonumber &\leq& C \left(\bar{t}-t\right)^{\frac{q}{2}-1}\int_0^{\bar{t}-t}ds\left(\int_{\RR^3}\int_{\RR^3}G(s,x-dy)G(s,x-dz)f(y-z)\right)^{\frac{q}{2}}\\
&&\nonumber \times\sup_{s\in[0,T],y,z\in
\RR^3}\EE\left((1+|u(s,y)|)^{\frac{q}{2}}(1+|u(s,z)|)^{\frac{q}{2}}\right)\\
&\leq& \nonumber
C(\bar{t}-t)^{\frac{q}{2}-1}\int_0^{\bar{t}-t}ds\left(
\int_{\RR^3}\int_{\RR^3}
 G (s,dy)G (s,dz)
f(y-z)\right)^{\frac{q}{2}}\\
\nonumber
&\leq&C(\bar{t}-t)^{\frac{q}{2}-1}\int_0^{\bar{t}-t}ds\left(
\int_{|z|\leq 2s}\frac{f(z)}{|z|}dz\right)^{\frac{q}{2}}\\
&\leq&C (\bar{t}-t)^{\frac{q}{2}-1}\int_0^{\bar{t}-t}
s^{\nu\frac{q}{2}}ds
 = C(\bar{t}-t)^{q\frac{\nu+1}{2}}. \label{T_4 1}
\end{eqnarray}
  For $T_{4, 2}$, for notational convenience we denote $\bar{t}-t$ by $h$. Applying Burkholder's inequality (see  Lemma \ref{Burkholder})  yields
\begin{eqnarray*}
 T_{4, 2  }
&\leq&C\EE\Big(\int_0^t\int_{\RR^3}\int_{\RR^3}\left(G (h+s,dy)-G (s,dy)\right)\left(G (h+s,dz)-G (s,dz)\right)\\
&&\times
f(y-z)\Sigma_{t,x}(s,y)\Sigma_{t,x}(s,z) ds\Big)^{\frac{q}{2}},
\end{eqnarray*}
where
  $\Sigma_{t,x}(s,y)=\sigma(t-s,x-y,u(t-s,x-y)) $.
  The integral with respect to the space variables $y$ and $z$  is actually taken on
the sphere $S^2$ in the three dimensional space because of the
structure of the fundamental solution $G$. We denote
$\xi=\frac{y}{|y|}$ and $\eta=\frac{z}{|z|}$ and  we recall that   $\sigma(d\xi)$ and $\sigma(d\eta)$ denote the uniform measure on $S^2$, so
\begin{eqnarray*}
&&G(s,dy)=\frac{1}{4\pi}s\sigma(d\xi),\\
&&G(s,dz)=\frac{1}{4\pi}s\sigma(d\eta).
\end{eqnarray*}
After some rearrangements similar to those made for  $Q$ in the proof of Theorem
\ref{space Holder Thm}, we can write
\begin{eqnarray*}
T_{4,2}=C\EE\Big(R_1+R_2+R_3+R_4\Big)^{\frac{q}{2}} \leq
C\sum_{i=1}^4\EE|R_i|^{\frac{q}{2}}\,,
\end{eqnarray*}
where
\begin{eqnarray*}
R_1&=&\int_0^t\int_{S^2\times S^2}(s+h)^2f\left((s+h)\xi-(s+h)\eta\right)\\
&&\times\left(\Sigma_{t,x}\left(s,(s+h)\xi\right)-\Sigma_{t,x}\left(s,s\xi\right)\right)\left(\Sigma_{t,x}\left(s,(s+h)\eta\right)-\Sigma_{t,x}\left(s,s\eta\right)\right)\sigma(d\xi) \sigma(d\eta) ds,\\
R_2&=&\int_0^t\int_{S^2\times S^2}\left((s+h)^2f\left((s+h)\xi-(s+h)\eta\right)-s(s+h)f\left(s\xi-(s+h)\eta\right)\right)\\
&&\times\left(\Sigma_{t,x}\left(s,(s+h)\eta\right)-\Sigma_{t,x}\left(s,s\eta\right)\right)\Sigma_{t,x}\left(s,s\xi\right)\sigma(d\xi)\sigma(d\eta)ds,\\
R_3&=&\int_0^t\int_{S^2\times S^2}\left((s+h)^2f\left((s+h)\xi-(s+h)\eta\right)-s(s+h)f\left((s+h)\xi-s\eta\right)\right)\\
&&\times\left(\Sigma_{t,x}\left(s,(s+h)\xi\right)-\Sigma_{t,x}\left(s,s\xi\right)\right)\Sigma_{t,x}\left(s,s\eta\right)\sigma(d\xi)\sigma(d\eta)ds,\\
R_4&=&\int_0^t\int_{S^2\times S^2}\Big((s+h)^2f\left((s+h)\xi-(s+h)\eta\right)-s(s+h)f\left(s\xi-(s+h)\eta\right)\\
&&-s(s+h)f\left((s+h)\xi-s\eta\right)+s^2f(s\xi-s\eta)\Big)\Sigma_{t,x}(s,s\xi)
\Sigma_{t,x}(s,s\eta)\sigma(d\xi)\sigma(d\eta)ds.
\end{eqnarray*}
We estimate each $\EE|R_i|^{\frac{q}{2}}$ separately.

For $\EE|R_1|^{\frac{q}{2}}$, using  H\"older's inequality, the
Lipschitz condition on  $\sigma$, the assumption on the  H\"older
continuity on  the space variable of $u$ (condition (2)), Lemma \ref{conv of G G lemma} and
condition \eref{H1}, we have
\begin{eqnarray}
 \nonumber \EE|R_1|^{\frac{q}{2}}
\nonumber &\leq&C h^{q\kappa} \int_0^t \left(\int_{S^2\times
S^2}(s+h)^2f\left((s+h)\xi-(s+h)\eta\right)\sigma(d\xi)\sigma(d\eta)\right)^{\frac{q}{2}} ds \\
\nonumber &=&C h^{q\kappa} \int_0^t\left(\int_{\RR^3}\int_{\RR^3}f(y-z)G(s+h,dy)G(s+h,dz)\right)^{\frac{q}{2}} ds \\
\nonumber &=&C h^{q\kappa} \int_0^t\left(\int_{|z|\leq 2(s+h)}\frac{f(z)}{|z|}dz\right)^{\frac{q}{2}} ds \\
&\leq&C h^{q\kappa}. \label{R_1}
\end{eqnarray}
In order to estimate $\EE|R_2|^{\frac{q}{2}}$,  we make the
decomposition
\begin{eqnarray*}
 \EE|R_2|^{\frac{q}{2}}
&\leq&C\EE\Big(\int_0^t\int_{S^2\times S^2}s(s+h) \left|f\left((s+h)\xi-(s+h)\eta\right)-f\left(s\xi-(s+h)\eta\right)\right|\\
&&\quad \times\left|\Sigma_{t,x}\left(s,(s+h)\eta\right)-\Sigma_{t,x}\left(s,s\eta\right)\right||\Sigma_{t,x}\left(s,s\xi\right)|\sigma(d\xi)\sigma(d\eta)ds\Big)^{\frac{q}{2}}\\
&&+C\EE\Big(\int_0^t\int_{S^2\times
S^2}h(s+h)f\left((s+h)\xi-(s+h)\eta\right) \\
&&\quad \times \left| \Sigma_{t,x}(s,s\xi)\right| \left|\Sigma_{t,x}\left(s,(s+h)\eta\right)-\Sigma_{t,x}\left(s,s\eta\right)\right|\sigma(d\xi)\sigma(d\eta)ds\Big)^{\frac{q}{2}}\\
&:=&R_2^1+R_2^2.
\end{eqnarray*}
For $R_2^1$,  using the H\"older inequality,  the Lipschitz  and
linear growth conditions on $\sigma$, the moments estimate \eref{moments estimate}, the assumption on the H\"older
continuity in the space variable of $u$  (condition (2)) and condition \eref{time cond 1}
with the change of variable $\eta \to -\eta$, we have
\begin{eqnarray*}
 R_2^1
&\leq&C\Big(\int_0^t \int_{S^2\times
S^2} s(s+h) \big|f\left((s+h)\xi-(s+h)\eta\right)\\
&& \quad-f\left(s\xi-(s+h)\eta\right)\big|\sigma(d\xi)\sigma(d\eta)ds\Big)^{\frac{q}{2}}h^{\frac{q\kappa}{2}}\\
&\leq&C\left(\int_0^t\int_{S^2\times
S^2}s \left|f\left((s+h)\xi+(s+h)\eta\right)-f\left(s\xi+(s+h)\eta\right)\right|\sigma(d\xi)\sigma(d\eta)ds\right)^{\frac{q}{2}}h^{\frac{q\kappa}{2}}\\
&\leq&Ch^{\frac{q\rho_1+q\kappa}{2}}.
\end{eqnarray*}
For $R_2^2$, by using  H\"older inequality's,  the Lipschitz condition   and
linear growth conditions on $\sigma$, the moments estimate \eref{moments estimate}, the assumption on the  H\"older
continuity in the space variable (condition (2))  and condition (1), we have
\begin{eqnarray*}
 R_2^2
&\leq&C\left(\int_0^t \int_{S^2\times
S^2}(s+h)f\left((s+h)\xi-(s+h)\eta\right)\sigma(d\xi)\sigma(d\eta)ds\right)^{\frac{q}{2}}h^{\frac{q+q\kappa}{2}}\\
&\leq&C\left(\int_0^t\frac{1}{s+h}
\int_{\RR^3}\int_{\RR^3}f(y-z)G(s+h,dy)G(s+h,dz)ds\right)^{\frac{q}{2}}h^{\frac{q+q\kappa}{2}}\\
&=&C\left(\int_0^t \frac{1}{s+h}\int_{|z|\leq 2(s+h)}\frac{
f(z)}{|z|}dz ds\right)^{\frac{q}{2}}h^{\frac{q+q\kappa}{2}}\\
&\leq&C h^{\frac{q+q\kappa}{2}}.
\end{eqnarray*}
Combining the estimates for $R_2^1$ and $R_2^2$, we have
\begin{equation}\label{R_2}
\EE|R_2|^{\frac{q}{2}}\leq C h^{\frac{q(\rho_1+\kappa)}{2}}.
\end{equation}
Similarly,
\begin{equation}\label{R_3}
\EE|R_3|^{\frac{q}{2}}\leq C h^{\frac{q(\rho_1+\kappa)}{2}}.
\end{equation}

For $R_4$,   using the  linear growth of $\sigma$, the moments
estimate \eref{moments estimate} and the change of
variable $\eta \to -\eta$, we have
\begin{eqnarray*}
 \EE|R_4|^{\frac{q}{2}}
&\leq&C\Big(\int_0^t\int_{S^2\times S^2}\big|(s+h)^2f\left((s+h)\xi-(s+h)\eta\right)-s(s+h)f\left(s\xi-(s+h)\eta\right)\\
&&-s(s+h)f\left((s+h)\xi-s\eta\right)+s^2f(s\xi-s\eta)\big|\sigma(d\xi)\sigma(d\eta)ds\Big)^{\frac{q}{2}}\\
&\leq&C\Big(\int_0^t\int_{S^2\times
S^2}s^2\big|f\left((s+h)\xi+(s+h)\eta\right)-f\left(s\xi+(s+h)\eta\right)\\
&&\ \ \ \ -f\left((s+h)\xi+s\eta\right)+f\left(s\xi+s\eta\right)\big|\sigma(d\xi)\sigma(d\eta)ds\Big)^{\frac{q}{2}}\\
&&+C\left(\int_0^t\int_{S^2\times
S^2}sh\left|f\left((s+h)\xi+(s+h)\eta\right)-f\left(s\xi+(s+h)\eta\right)\right|\sigma(d\xi)\sigma(d\eta)ds\right)^{\frac{q}{2}}\\
&&+C\left(\int_0^t\int_{S^2\times
S^2}sh\left|f\left((s+h)\xi+(s+h)\eta\right)-f\left((s+h)\xi+s\eta\right)\right|\sigma(d\xi)\sigma(d\eta)ds\right)^{\frac{q}{2}}\\
&&+C\left(\int_0^t\int_{S^2\times
S^2}h^2f\left((s+h)\xi-(s+h)\eta\right)\sigma(d\xi)\sigma(d\eta)ds\right)^{\frac{q}{2}}\\
&:=&R_4^1+R_4^2+R_4^3+R_4^4.
\end{eqnarray*}
For $R_4^1$ condition \eref{time cond 2} yields
\begin{equation*}
R_4^1\leq  C h^{\frac{q\rho_2}{2}}.
\end{equation*}
For $R_4^2$ and $R_4^3$  applying condition  \eref{time cond 1}  and we obtain
\begin{equation*}
R_4^2\leq C h^{\frac{q\rho_1+q}{2}}, \quad R_4^3\leq C h^{\frac{q\rho_1+q}{2}}.
\end{equation*}
For $R_4^4$,  condition (1) allows us to write
\begin{eqnarray*}
R_4^4&=& C h^q\left(\int_0^t \int_{|z|\leq 2(s+h)}\frac{1}{(s+h)^2}\frac{f(z)}{|z|}dz ds\right)^{\frac{q}{2}}\\
&\leq&C h^q \left(\int_0^t (s+h)^{-2+\nu} ds \right)^{\frac{q}{2}}.
\end{eqnarray*}
When $\nu < 1$, $R_4^4 \leq C h^{\frac{q(\nu+1)}{2}}$, when $\nu=1$,
$R_4^4\leq C h^q (\log (T+h)-\log h)\leq C h^{q(1-\varepsilon)}$ for
any $\varepsilon > 0$.

Combining the estimates for $R_4^1$, $R_4^2$, $R_4^3$, $R_4^4$, we
have
\begin{equation}\label{R_4}
\EE|R_4|^{\frac{q}{2}}\leq C
(h^{\frac{q\rho_2}{2}}+h^{\frac{q\rho_1+q}{2}}+h^{q\frac{\nu+1}{2}}+h^{q(1-\varepsilon)})\,,
\end{equation}
for any $\varepsilon > 0$.  By \eref{T_4 1}, \eref{R_1}, \eref{R_2},
\eref{R_3} and \eref{R_4}, we conclude that
\begin{equation}\label{T_4 final}
T_4 \leq Ch^{q\rho}\,,
\end{equation}
where $0<\rho <\min(\frac{\nu+1}{2},\frac{\rho_1+\kappa}{2},\frac{\rho_2}{2},\kappa)$.
From the proof it is easy to see that the constant $C$ in the above
expression does not depend on $x$. Then we combine the estimates of
\eref{T_1+T_2}, \eref{T_3} and \eref{T_4 final} to obtain
\begin{equation*}
\sup_{x \in \RR^3}\EE|u(t,x)-u(\bar{t},x)|^q\leq C |\bar{t}-t|^{q\kappa^{\prime}}\,,
\end{equation*}
where $\kappa^{\prime} \in \left(0, \min(\gamma_1, \gamma_2,
\frac{\nu+1}{2}, \frac{\rho_1+\kappa}{2}, \frac{\rho_2}{2},
\kappa)\right)$.
\end{proof}

\medskip
An application of Kolmogorov's continuity criteria leads to the following H\"older continuity result in the space an time variables.

\begin{corollary}
Let $u$ be the solution to Equation \eref{our SPDE}.  Assume
conditions (a) and (b) in Theorem \ref{space Holder Thm no Fourier}. Suppose that condition (c) of Theorem 3.1 or condition holds (a) of Theorem 3.2 hold. Set  $\kappa_1= \min(\gamma_1,\gamma_2,\gamma,\frac{\gamma^{\prime}}{2})$ in the first case and   $\kappa_1= \min(\gamma_1,\gamma_2,\gamma)$ in the second case. Suppose also that conditions (1), (2) and (3) of Theorem 3.2 hold. Set $\kappa_2=  \min(\gamma_1,\gamma_2,\kappa_1,\frac{\nu+1}{2},\frac{\rho_1+\kappa_1}{2},\frac{\rho_2}{2})$. Then, for any $\kappa< \kappa_1$ and $\kappa'<\kappa_2$ there exists a version of the process $u$ which is locally H\"older continuous of order $\kappa$ in the space variable and of order $\kappa'$ in the time variable. That is, for any bounded rectangle $I\subset \RR^3$ we can find a random variable $K_{\kappa, \kappa',I}$ such that
\[
|u(t,x)- u(s,y)| \le K_{\kappa, \kappa',I} \left( |t-s|^{\kappa'} + |x-y|^{\kappa} \right)
\]
for all $s,t\in [0,T]$ and $x,y\in I$.
 \end{corollary}

\section{Examples}
In this section, we give some examples of covariance functions $f$
satisfying the conditions in the previous theorems.
\subsection{Example 1}
\begin{prop}
Let $f(x)=(\rho*\frac{1}{|\cdot|^{\beta}})(x))$, where $\rho(x)$ is
a nonnegative Schwartz function defined in $\RR^3$ such that
$(\mathcal{F}^{-1}\rho)(\xi)\geq 0$ (for example, $\rho(x)=e^{-|x|^2}$) and $0 < \beta <3$. Then condition (a) of  Theorem \ref{space Holder Thm} holds
for $0 < \gamma < \min(\frac{3-\beta}{2}, 1)$.
\end{prop}
\begin{proof}
Since $\rho$ is a Schwartz function, the spectral measure can be
explicitly expressed as
\begin{equation}
\mu(d\xi)=C(\mathcal{F}^{-1}\rho)(\xi)\frac{1}{|\xi|^{3-\beta}}d\xi
\end{equation}
for some constant $C$ which only depends on $\beta$. So
\begin{eqnarray*}
\int_{\RR^3}\frac{\mu(d\xi)}{1+|\xi|^{2-2\gamma}}=\int_{\RR^3}C(\mathcal{F}^{-1}\rho)(\xi)\frac{1}{|\xi|^{3-\beta}}\frac{1}{1+|\xi|^{2-2\gamma}}d\xi<\infty
\end{eqnarray*}
since $(\mathcal{F}^{-1}\rho)(\xi)$ is rapidly decreasing. The
Fourier transform of the tempered measure $|\xi|^{2\gamma}\mu(d\xi)$
is
\begin{eqnarray*}
(\mathcal{F}(|\xi|^{2\gamma}\mu(d\xi)))(x)=C
\mathcal{F}((\mathcal{F}^{-1}\rho)(\xi)\frac{1}{|\xi|^{3-\beta-2\gamma}}d\xi)(x)=(\rho*\frac{1}{|\cdot|^{\beta+2\gamma}})(x)\,,
\end{eqnarray*}
which is a nonnegative locally integrable function, and the last
equality holds when $3-\beta-2\gamma>0$.  So the condition  (a)  is satisfied.
\end{proof}

\subsection{The Riesz kernel}

Before giving next example, we recall  two inequalities  from Dalang and Sanz-Sol\'e  \cite{dalang4}.

Let $d$ be a positive integer. Let  $\xi$, $\eta$ be two unit vectors in $\RR^d$ and let  $u$ be any
point in $\RR^d$.  Suppose $a$, $b$ are positive numbers with $a+b
\in (0,d)$. Then we have for any $h \in \RR$
\begin{equation}\label{increment a}
|u+h\xi|^{a+b-d}-|u|^{a+b-d}=|h|^b\int_{\RR^d}dw
|u-hw|^{a-d}(|w+\xi|^{b-d}-|w|^{b-d}),
\end{equation}
and
\begin{eqnarray}
&&\left||u+h\xi+h\eta|^{a+b-d}-|u+h\xi|^{a+b-d}-|u+h\eta|^{a+b-d}+|u|^{a+b-d}\right| \notag\\
&&\!\!\!\!\! \leq |h|^b \int_{\RR^d}dw
|u-hw|^{a-d}\left||w+h\xi+h\eta|^{b-d}-|w+h\xi|^{b-d}-|w+h\eta|^{b-d}+|w|^{b-d}\right|.
\label{increment e}
\end{eqnarray}

\begin{prop}
Let $f(x)=|x|^{-\beta}$, $0<\beta<2$. Then
 $f$ satisfies condition (a)  in Theorem \ref{space Holder Thm} for any $\gamma \in
(0,\frac{2-\beta}{2})$ and $f$ also satisfies  conditions (1),
\eref{time cond 1} and \eref{time cond 2} in Theorem \ref{time holder Thm}
for  $\nu = 2-\beta$,  any $0 < \rho_1 < \min(2-\beta,
1)$ and  $0 < \rho_2 < 2-\beta$.
\end{prop}

\begin{proof}
Let us first check condition (a) in Theorem \ref{space Holder Thm}. Since
$f(x)=|x|^{-\beta}$, we have $\mu(d\xi)=C|\xi|^{-3+\beta}d\xi$.  Then it is
easy to see that
\begin{equation*}
\int_{\RR^3}\frac{\mu(d\xi)}{1+|\xi|^{2-2\gamma}} < \infty,
\end{equation*}
since $0<\gamma < \frac{2-\beta}{2}$, and we have
\begin{eqnarray*}
\mathcal{F}\left(|\xi|^{2\gamma}\mu(d\xi)\right)(x)=C\mathcal{F}(|\xi|^{-3+\beta+2\gamma}d\xi)(x)=C|x|^{-(\beta+2\gamma)}
\end{eqnarray*}
for some positive constant $C$, so the above expression is
nonnegative. So,  condition (a) in Theorem \ref{space Holder Thm} holds.

To verify   condition (1) in Theorem \ref{time holder Thm}, we notice
\begin{eqnarray*}
\int_{|z|\leq h}\frac{f(z)}{|z|}dz=\int_{|z|\leq
h}|z|^{-\beta-1}dz=C h^{2-\beta}.
\end{eqnarray*}
So condition (1) in Theorem \ref{time holder Thm} is satisfied with
$\nu = 2-\beta$.

We turn to   condition \eref{time cond 1}. We apply \eref{increment a} with
$b=\rho_1<\min((2-\beta),1)$, $d=3$, $a=3-\rho_1-\beta$,
 $u=s(\xi+\eta)+h\eta$ to get
\begin{eqnarray*}
&&\int_0^T \int_{S^2}\int_{S^2}s
\left|f\left(s(\xi+\eta)+h(\xi+\eta)\right)-f\left(s(\xi+\eta)+h\eta\right)\right|\sigma(d\xi)\sigma(d\eta)ds\\
&\leq&h^{\rho_1}\int_0^T \int_{S^2 \times S^2} s \int_{\RR^3} dw
\left|s\xi+(s+h)\eta-hw\right|^{-\rho_1-\beta}\left||w+\xi|^{\rho_1-3}-|w|^{\rho_1-3}\right|\sigma(d\xi)\sigma(d\eta)ds\\
&\leq&h^{\rho_1}\int_0^T \int_{S^2 \times S^2} s \int_{|w|\leq 3} dw
\left|s\xi+(s+h)\eta-hw\right|^{-\rho_1-\beta}|w+\xi|^{\rho_1-3}\sigma(d\xi)\sigma(d\eta)ds\\
&&+h^{\rho_1}\int_0^T \int_{S^2 \times S^2} s \int_{|w|\leq 3} dw
\left|s\xi+(s+h)\eta-hw\right|^{-\rho_1-\beta}|w|^{\rho_1-3}\sigma(d\xi)\sigma(d\eta)ds\\
&&+h^{\rho_1}\int_0^T \int_{S^2 \times S^2} s \int_{|w|>3} dw
\left|s\xi+(s+h)\eta-hw\right|^{-\rho_1-\beta}\left||w+\xi|^{\rho_1-3}-|w|^{\rho_1-3}\right|\sigma(d\xi)\sigma(d\eta)ds\\
&:=&h^{\rho_1}(I_1+I_2+I_3).
\end{eqnarray*}
For $I_1$, making the change of variable $w+h \rightarrow w$, using the
Fourier transform (see Lemma \ref{int identity of GGf}) and noting
that $I_1$ is real positive, we can write:
\begin{eqnarray*}
 I_1 &\leq&\int_0^T \int_{S^2\times S^2} s \int_{|w|\leq 4}
\left|(s+h)\xi+(s+h)\eta-hw\right|^{-\rho_1-\beta}|w|^{\rho_1-3}dw\sigma(d\xi)\sigma(d\eta)ds\\
&=&C\int_0^T \int_{|w|\leq 4}\int_{\RR^3\times
\RR^3}\frac{s}{(s+h)^2}|y+z-hw|^{-\rho_1-\beta}G(s+h,dy)G(s+h,dz)|w|^{\rho_1-3}dw
ds \\
&= &C\int_0^T \int_{\RR^3}
\frac{s}{(s+h)^2}\frac{\left(\sin(s+h)|\xi|\right)^2}{|\xi|^2}|\xi|^{-3+\rho_1+\beta}e^{i\langle
\xi,hw\rangle}d\xi ds \int_{|w|\leq 4} |w|^{\rho_1-3}dw.
\end{eqnarray*}
Then using the change of variable $(s+h)\xi=\eta$ and the bound
$|e^{i\langle \xi, hw\rangle}|\leq 1$, by direct calculation we see
that $I_1<\infty$.

For $I_2$, we do the same calculation, but we do not need the
change of variable for $w$. Let $2 \varepsilon < 2-\beta-\rho_1$, then
\begin{eqnarray*}
 I_2 &\leq&\int_0^T \int_{\RR^3}
\frac{1}{s+h}\frac{\left|\sin(s+h)|\xi|\sin
s|\xi|\right|}{|\xi|^2}|\xi|^{-3+\rho_1+\beta}d\xi ds \int_{|w|\leq 3}
|w|^{\rho_1-3}dw\\
&\leq&C \int_0^T \int_{|\xi|\leq 1} s|\xi|^{-3+\rho_1+\beta}d\xi ds
+C \int_0^T \int_{|\xi|>1}\frac{1}{s+h}\frac{(s+h)^{\varepsilon}
s^{\varepsilon}
|\xi|^{2\varepsilon}}{|\xi|^2}|\xi|^{-3+\rho_1+\beta}d\xi ds
\end{eqnarray*}
which is finite by direct calculation.

For $I_3$ we can write

\begin{eqnarray*}
I_3 &=&\int_0^T \int_{|w|>3} \int_{S^2\times S^2} s
\left|s\xi+(s+h)\eta-hw\right|^{-\rho_1-\beta}\sigma(d\xi)\sigma(d\eta)dw\left|\int_0^1
\frac{d}{d\lambda}|w+\lambda \xi|^{\rho_1-3}d\lambda\right| ds \\
&\leq& C\int_0^T  \int_{|w|>3}\int_{S^2\times S^2} s
\left|s\xi+(s+h)\eta-hw\right|^{-\rho_1-\beta}\sigma(d\xi)\sigma(d\eta)
(\int_0^1 |w|^{\rho_1-4}d\lambda ) dw ds \,,
\end{eqnarray*}
where the inequality holds because $|w|> 3$, $0\leq \lambda \leq 1$ and
$|\xi|=1$. We can show that $I_3<\infty$ similarly to the proof for $I_2$ using the
fact that $\int_{|w|> 3}|w|^{\rho_1-4}dw < \infty $ since
$\rho_1<1$.

It is easy to see that $I_1$ $I_2$ $I_3$ are finite uniformly
 for $0<h \leq 1$. Therefore,   condition \eref{time cond 1} is satisfied with $0 < \rho_1 < \min
(2-\beta,1)$.

For  condition \eref{time cond 2}, applying \eref{increment e}, with
$d=3$, $b=\rho_2  < 2-\beta$, $a=3-\rho_2-\beta$,
 $u=s(\xi+\eta)$, yields
\begin{eqnarray*}
&&\int_0^T \int_{S^2}\int_{S^2}
\left|f\left(s(\xi+\eta)+h(\xi+\eta)\right)-f\left(s(\xi+\eta)+h\xi\right)-f\left(s(\xi+\eta)+h\eta\right)+f\left(s(\xi+\eta)\right)\right|\\
&&\times s^2\sigma(d\xi)\sigma(d\eta)\\
&\leq& h^{\rho_2} \int_0^T\int_{S^2\times
S^2}s^2\int_{\RR^3}\left|s(\xi+\eta)-hw\right|^{-\rho_2-\beta}\\
&&\qquad\times\left||w+\xi+\eta|^{\rho_2-3}-|w+\xi|^{\rho_2-3}-|w+\eta|^{\rho_2-3}+|w|^{\rho_2-3}\right|dw\sigma(d\xi)\sigma(d\eta)ds\\
&\leq& h^{\rho_2} \Big(\int_0^T \int_{S^2\times S^2}s^2\int_{|w|\leq
3}\left|s(\xi+\eta)-hw\right|^{-\rho_2-\beta}|w+\xi+\eta|^{\rho_2-3}dw\sigma(d\xi)\sigma(d\eta)ds\\
&&+\int_0^T \int_{S^2\times S^2}s^2\int_{|w|\leq
3}\left|s(\xi+\eta)-hw\right|^{-\rho_2-\beta}|w+\xi|^{\rho_2-3}dw\sigma(d\xi)\sigma(d\eta)ds\\
&&+\int_0^T \int_{S^2\times S^2}s^2\int_{|w|\leq
3}\left|s(\xi+\eta)-hw\right|^{-\rho_2-\beta}|w+\eta|^{\rho_2-3}dw\sigma(d\xi)\sigma(d\eta)ds\\
&&+\int_0^T \int_{S^2\times S^2}s^2\int_{|w|\leq
3}\left|s(\xi+\eta)-hw\right|^{-\rho_2-\beta}|w|^{\rho_2-3}dw\sigma(d\xi)\sigma(d\eta)ds\\
&&+\int_0^T \int_{S^2\times S^2}s^2\int_{|w|>
3}\left|s(\xi+\eta)-hw\right|^{-\rho_2-\beta}\\
&&\qquad\times\left||w+\xi+\eta|^{\rho_2-3}- |w+\xi|^{\rho_2-3}- |w+\eta|^{\rho_2-3}+ |w|^{\rho_2-3}\right|dw\sigma(d\xi)\sigma(d\eta)ds\Big)\\
&:= & h^{\rho_2}(\sum_{i=1}^5 L_i).
\end{eqnarray*}
For $L_i$, $i=1,2,3,4$, we can proceed exactly in the same way as
for the integrals $I_1$, $I_2$ above. For $L_5$, we can express
\begin{equation*}
|w+\xi+\eta|^{\rho_2-3}- |w+\xi|^{\rho_2-3}- |w+\eta|^{\rho_2-3}+
|w|^{\rho_2-3}=\int_0^1\int_0^1 \frac{\partial^2}{\partial \lambda
\partial \mu}|w+\lambda \xi+\mu \eta|^{\rho_2-3}d\lambda d\mu\,,
\end{equation*}
and since $|w|>3$, $|\eta|=1$, it is easy to see that
\begin{equation*}
\left|\frac{\partial^2}{\partial \lambda \partial
\mu}|w+\lambda \xi+\mu \eta|^{\rho_2-3}\right|\leq C
|w|^{\rho_2-5}.
\end{equation*}
So $\int_{|w|>3}|w|^{\rho_2-5}dw$ is finite, and $L_5$ is finite, by
the same argument as for $I_3$.

So  condition \eref{time cond 2} is satisfied with $0 < \rho_2 <
2-\beta$. This completes the proof.
\end{proof}

\medskip
Notice  that, with the notation of Corollary 4.2, for the Riesz kernel we can take $\kappa_1=\kappa_2 <\frac {2-\beta }2$, and we deduce  the local H\"older continuity  of the solution $u$ in space and time  variables of order $\kappa < \min(\gamma_1, \gamma_2, \frac {2-\beta}2 )$. In this way we recover the result by Dalang and Sanz-Sol\'e   \cite{dalang4}.

\subsection{The Bessel Kernel}
\begin{prop}
Let $f$ be the Bessel kernel given by
\begin{equation}\label{Bessel ker}
f(x)=\int_0^{\infty}w^{\frac{\alpha-5}{2}}e^{-w}e^{-\frac{|x|^2}{4w}}dw
\end{equation}
for some $\alpha>1$. Then $f$ satisfies \eref{H1}, \eref{sp cond
1}, \eref{sp cond 2}, \eref{time cond 1}, \eref{time cond 2} and
 condition (1) in Theorem \ref{time holder Thm} for any $0 < \gamma,
\rho_1, \nu < \min(\alpha-1, 1)$ and $0 < \gamma^{\prime}, \rho_2 <
\min(\alpha-1,2)$.
\end{prop}
\begin{proof}
First let us check condition (\ref{H1}). We have
\begin{eqnarray*}
\int_{|x|\leq
1}\frac{f(x)}{|x|}dx=\int_0^{\infty}w^{\frac{\alpha-5}{2}}e^{-w}\int_{|x|\leq
1}\frac{e^{-\frac{|x|^2}{4w}}}{|x|}dxdw.
\end{eqnarray*}
The change of variable $x=2\sqrt{w}y$ gives
\begin{eqnarray*}
\int_{|x|\leq
1}\frac{f(x)}{|x|}dx=4\int_0^{\infty}w^{\frac{\alpha-3}{2}}e^{-w}\int_{|y|\leq
\frac{1}{2\sqrt{w}}}\frac{e^{-|y|^2}}{|y|}dy\leq
C\int_0^{\infty}w^{\frac{\alpha-3}{2}}e^{-w}dw<\infty
\end{eqnarray*}
because $\alpha>1$. To check condition \eref{sp cond 1}, we note
that for $a,b \geq 0$, we have $|e^{-a}-e^{-b}|\leq
|a-b|^{\gamma}(e^{-a}\vee e^{-b})$, for any $0 \le \gamma \leq 1$. So
\begin{eqnarray*}
|e^{-\frac{|z+y|^2}{4w}}- e^{-\frac{|z|^2}{4w}}|&\leq&
(\frac{1}{4w})^{\gamma}\left||z+y|^2-|z|^2\right|^{\gamma}\left(e^{-\frac{|z+y|^2}{4w}}\vee
e^{-\frac{|z|^2}{4w}}\right)\\
&\leq& C
|y|^{\gamma}(|z+y|^{\gamma}+|z|^{\gamma})\frac{1}{w^{\gamma}}\left(e^{-\frac{|z+y|^2}{4w}}+e^{-\frac{|z|^2}{4w}}\right).
\end{eqnarray*}
 As a consequence
\begin{eqnarray*}
&&\int_{|z|\leq
2T}\frac{|f(z+y) -f(z)|}{|z|}dz\\
&\leq&|y|^{\gamma}\int_0^{\infty}w^{\frac{\alpha-5}{2}-\gamma}e^{-w}
\Big(\int_{|z|\leq
2T}(|z+y|^{\gamma}+|z|^{\gamma})\frac{e^{-\frac{|z+y|^2}{4w}}}{|z|}dz
\\
&&+\int_{|z|\leq
2T}(|z+y|^{\gamma}+|z|^{\gamma})\frac{e^{-\frac{|z|^2}{4w}}}{|z|}dz\Big)dw\\
&:=&|y|^{\gamma}\int_0^{\infty}w^{\frac{\alpha-5}{2}-\gamma}e^{-w}\left(I(y)+J(y)\right)dw.
\end{eqnarray*}
For the integral $I(y)$, with the change of variable $z=\sqrt{w}x-y$,
we have
\begin{eqnarray*}
I(y)&=&\int_{|x-\frac{y}{\sqrt{w}}|\leq
\frac{2T}{\sqrt{w}}}w^{\frac{\gamma+2}{2}}\left(\frac{|x|^{\gamma}}{|x-\frac{y}{\sqrt{w}}|}+|x-\frac{y}{\sqrt{w}}|^{\gamma-1}\right)e^{-\frac{|x|^2}{4}}dx\\
&\leq&\int_{\RR^3}w^{\frac{\gamma+2}{2}}\left(\frac{|x|^{\gamma}}{|x-\frac{y}{\sqrt{w}}|}+|x-\frac{y}{\sqrt{w}}|^{\gamma-1}\right)e^{-\frac{|x|^2}{4}}dx\\
&\leq&C w^{\frac{\gamma+2}{2}}\,,
\end{eqnarray*}
where the last inequality follows from the fact that
$|x|^{\gamma}e^{-\frac{|x|^2}{4}}\leq C e^{-\frac{|x|^2}{8}}$ and
Lemma 17 in \cite{Nualart Rovira Tindel}. The term  $J(y)$  can be
estimated in the same way  using the change of variable $z=\sqrt{w}y$,
and we have
\begin{eqnarray*}
J(y)\leq C w^{\frac{\gamma+2}{2}}.
\end{eqnarray*}
Hence,
\begin{eqnarray*}
\int_{|z|\leq 2T}\frac{|f(z+y)-f(z)|}{|z|}dz\leq C |y|^{\gamma}
\int_0^{\infty}w^{\frac{\alpha-\gamma-3}{2}}e^{-w}dw\leq C
|y|^{\gamma}
\end{eqnarray*}
for any $0<\gamma < \alpha-1$. So condition \eref{sp cond 1} is
satisfied with $0< \gamma < \min (\alpha-1,1)$.

To check  condition \eref{sp cond 2}, note that
\begin{equation*}
|f(z+y)+f(z-y)-2f(z)|=\left|\int_0^1\int_0^1\frac{\partial^2}{\partial
\lambda \partial \mu}f(z-(\lambda-\mu)y)d\lambda d\mu\right|.
\end{equation*}
So we have
\begin{eqnarray*}
&&\left|e^{-\frac{|z+y|^2}{4w}}+e^{-\frac{|z-y|^2}{4w}}-2e^{-\frac{|z|^2}{4w}}\right|\\
&\leq& \int_0^1\int_0^1 \left(e^{-\frac{|z-\lambda y+\mu y|^2}{4w}}\frac{1}{4w^2}\langle z-\lambda y+\mu y, y\rangle^2+e^{-\frac{|z-\lambda y+\mu y|^2}{4w}}\frac{1}{2w}|y|^2\right)d\lambda d\mu\\
&\leq&\int_0^1\int_0^1 e^{-\frac{|z-\lambda y+\mu y|^2}{4w}}\left(\frac{1}{4w^2}|z-\lambda y+\mu y|^2|y|^2+\frac{1}{2w}|y|^2\right)d\lambda d\mu\\
&\leq&C\int_0^1\int_0^1\frac{|y|^2}{w}e^{-\frac{|z-(\lambda-\mu)y|^2}{8w}}d\lambda
d\mu.
\end{eqnarray*}
Here we have used the fact that $x^2e^{-x^2}\leq C
e^{-\frac{x^2}{2}}$. By considering the cases $\frac{|y|}{\sqrt{w}}\leq
1$ and $\frac{|y|}{\sqrt{w}}>1$, we obtain
\begin{eqnarray*}
&&\left|e^{-\frac{|z+y|^2}{4w}}+e^{-\frac{|z-y|^2}{4w}}-2e^{-\frac{|z|^2}{4w}}\right|\\
&\leq&C\frac{|y|^{\gamma'}}{w^{\gamma'/2}}\int_0^1\int_0^1\left(e^{-\frac{|z-(\lambda-\mu)y|^2}{8w}}+e^{-\frac{|z+y|^2}{4w}}+e^{-\frac{|z-y|^2}{4w}}+2e^{-\frac{|z|^2}{4w}}\right)d\lambda
d\mu\,
\end{eqnarray*}
for any $0 \le \gamma^{\prime}\leq 2$. So we have
\begin{eqnarray*}
&&\int_{|z|\leq 2T}\frac{|f(z+y)+f(z-y)-2f(z)|}{|z|}dz\\
&\leq& C |y|^{\gamma^{\prime}}\int_{|z|\leq 2T}\int_0^{\infty}dw
\int_0^1 d\lambda \int_0^1 d\mu
w^{\frac{\alpha-5-\gamma'}{2}}e^{-w}\\
&& \times \left(e^{-\frac{|z-(\lambda-\mu)y|^2}{8w}}+e^{-\frac{|z+y|^2}{4w}}+e^{-\frac{|z-y|^2}{4w}}+2e^{-\frac{|z|^2}{4w}}\right)\frac{1}{|z|}dz.
\end{eqnarray*}
By Lemma 17 in \cite{Nualart Rovira Tindel}, we can write
\begin{eqnarray*}
\int_{|z|\leq
2T}\left(e^{-\frac{|z-(\lambda-\mu)y|^2}{8w}}+e^{-\frac{|z+y|^2}{4w}}+e^{-\frac{|z-y|^2}{4w}}+2e^{-\frac{|z|^2}{4w}}\right)\frac{1}{|z|}dz\leq
C w\,,
\end{eqnarray*}
where the constant $C$ does not depend on $y, \lambda, \mu$. Therefore,
\begin{eqnarray*}
\int_{|z|\leq 2T}\frac{|f(z+y)+f(z-y)-2f(z)|}{|z|}dz\leq C
|y|^{\gamma^{\prime}}\int_0^{\infty}w^{\frac{\alpha-3-\gamma'}{2}}e^{-w}dw\leq
C|y|^{\gamma^{\prime}}
\end{eqnarray*}
for any $0<\gamma^{\prime}<\min(\alpha-1, 2)$. As a consequence,   condition
\eref{sp cond 2} is satisfied with $0< \gamma^{\prime} <
\min(\alpha-1,2)$.

 To check   condition
(1) in Theorem \ref{time holder Thm} holds we compute
\[
\int_{|x| \le h} \frac{e^{-\frac {|x|^2}{ 4w}}}{|x|} dx=
4\pi \int_0^h e^{ -\frac {r^2}{4w}} rdr= 8\pi w \left(1- e^{-\frac {h^2}{4w}} \right)
\le C h^\nu  w^{1-\frac \nu 2}
\]
for any $0\le \nu \le 1$. This implies that
\[
\int_{|x|\leq h}\frac{f(x)}{|x|}dx \leq C
h^{\nu}\int_0^{\infty}w^{\frac{\alpha-3-\nu}{2}}e^{-w}dw
\leq C h^{\nu}
\]
for any $0< \nu <\min(\alpha-1,1)$.  So  condition (1) in Theorem
\ref{time holder Thm} is satisfied with $0< \nu <
\min(\alpha-1,1)$.

To check the condition \eref{time cond 1}, first we note that
\begin{eqnarray*}
&&\left|e^{-\frac{|x+h\xi|^2}{4w}}-e^{-\frac{|x|^2}{4w}}\right|
=\left|\int_0^1 \frac{d}{d\lambda}e^{-\frac{|x+\lambda
h\xi|^2}{4w}}d\lambda\right| =\left|\int_0^1 e^{-\frac{|x+\lambda
h\xi|^2}{4w}}\frac{\langle x+\lambda h\xi, h\xi\rangle}{2w}d
\lambda\right|\\
&\leq&C \int_0^1 e^{-\frac{|x+\lambda h\xi|^2}{4w}} \frac{|x+\lambda
h \xi|}{\sqrt{w}}\frac{h}{\sqrt{w}}d\lambda\leq C \int_0^1
e^{-\frac{|x+\lambda h\xi|^2}{8w}}\frac{h}{\sqrt{w}}d\lambda\,,
\end{eqnarray*}
where we have used the fact that $|x|e^{-x^2}\leq C
e^{-\frac{x^2}{2}}$. By considering the cases $\frac{h}{\sqrt{w}}
\leq 1$ and $\frac{h}{\sqrt{w}} > 1$, we can write
\begin{eqnarray*}
\left|e^{-\frac{|x+h\xi|^2}{4w}}-e^{-\frac{|x|^2}{4w}}\right|\leq C
\left(\frac{h}{\sqrt{w}}\right)^{\rho_1}\int_0^1\left(e^{-\frac{|x+\lambda
h\xi|^2}{8w}}+e^{-\frac{|x+h\xi|^2}{4w}}+e^{-\frac{|x|^2}{4w}}\right)d\lambda
\end{eqnarray*}
for any $\rho_1 \in [0,1]$. So we have
\begin{eqnarray*}
|f(x+h\xi)-f(x)|\leq C h^{\rho_1} \int_0^1 \int_0^{\infty}
w^{\frac{\alpha-5-\rho_1}{2}}e^{-w}\left(e^{-\frac{|x+\lambda
h\xi|^2}{8w}}+e^{-\frac{|x+h\xi|^2}{4w}}+e^{-\frac{|x|^2}{4w}}\right)dw
d\lambda.
\end{eqnarray*}
Therefore, for any $\rho_1 \in [0,1]$.
\begin{eqnarray*}
&&\int_0^T \int_{S^2}\int_{S^2}
\left|f\left(s(\xi+\eta)+h(\xi+\eta)\right)-f\left(s(\xi+\eta)+h\eta\right)\right|s\sigma(d\xi)\sigma(d\eta)ds\\
&\leq&C h^{\rho_1} \int_0^1 \int_0^{\infty}
w^{\frac{\alpha-5-\rho_1}2}   e^{-w}\int_0^T
\int_{S^2}\int_{S^2}\left(e^{-\frac{|(s+\lambda
h)\xi+(s+h)\eta|^2}{8w}}+e^{-\frac{|(s+h)(\xi+\eta)|^2}{4w}}+e^{-\frac{|s\xi+(s+h)\eta|^2}{4w}}\right)\\
&&\times s\sigma(d\xi)\sigma(d\eta)ds dw d\lambda.
\end{eqnarray*}
We claim that this quantity is bounded by $Ch^{\rho_1}$ if $0<\rho_1 < \min(\alpha-1,1)$.
To show this claim, we first estimate the quantity
\begin{equation*}
I_1:=\int_0^T \int_{S^2}\int_{S^2} s e^{-\frac{|(s+\lambda
h)\xi+(s+h)\eta|^2}{8w}}\si(d\xi)\si(d\eta)ds.
\end{equation*}
Using the Fourier transform (see Lemma \ref{int identity of GGf}), the
change of variables  $\xi\sqrt{w}=\eta$, and
taking $0< \varepsilon < 1$  we obtain
\begin{eqnarray*}
I_1&=&C\int_0^T \int_{\RR^3}\frac{s}{(s+h)(s+\lambda
h)}w^{\frac{3}{2}}e^{-2 w |\xi|^2}\frac{\sin
(s+h)|\xi|}{|\xi|}\frac{\sin (s+\lambda h)|\xi|}{|\xi|}d\xi ds\\
&\leq & C  w^{\frac 32}     \int_0^T s(s+h)^{\varepsilon -1} (s+\lambda h)^{\varepsilon -1} ds
\int_{\RR^3}  e^{-2 w |\xi|^2} |\xi|^{2\varepsilon-2} d\xi \\
&\le&  C w^{1-\varepsilon} \int_{\RR^3}  e^{-2  |\eta|^2} |\eta|^{2\varepsilon-2} d\eta
\le C w^{1-\varepsilon}.
\end{eqnarray*}
Similarly, we have
\begin{equation*}
\int_0^T \int_{S^2}\int_{S^2} s e^{-\frac{|(s+
h)\xi+(s+h)\eta|^2}{4w}}\si(d\xi)\si(d\eta)ds \leq C
w^{1-\varepsilon}
\end{equation*}
and
\begin{equation*}
\int_0^T \int_{S^2}\int_{S^2} s
e^{-\frac{|s\xi+(s+h)\eta|^2}{4w}}\si(d\xi)\si(d\eta)ds \leq C
w^{1-\varepsilon} .
\end{equation*}
Therefore,
\begin{eqnarray*}
&&\sup_{0<h\leq 1}\int_0^1 \int_0^{\infty}
w^{\frac{\alpha-5-\rho_1}{2}}e^{-w}\int_0^T
\int_{S^2}\int_{S^2}\left(e^{-\frac{|(s+\lambda
h)\xi+(s+h)\eta|^2}{8w}}+e^{-\frac{|(s+h)(\xi+\eta)|^2}{4w}}+e^{-\frac{|s\xi+(s+h)\eta|^2}{4w}}\right)\\
&&\times s\sigma(d\xi)\sigma(d\eta)ds dw d\lambda\\
 &\leq&C \int_0^1
\int_0^{\infty}
w^{\frac{\alpha-5-\rho_1}{2}+1-\varepsilon}e^{-w}dw
d\lambda < \infty,
\end{eqnarray*}
and \eref{time cond 1} is  satisfied with $0<\rho_1 <
\min(\alpha-1,1))$.

To check  condition \eref{time cond 2}, we note that
\begin{eqnarray*}
&&\left|e^{-\frac{|x+h\xi+h\eta|^2}{4w}}-e^{-\frac{|x+h\xi|^2}{4w}}-e^{-\frac{|x+h\eta|^2}{4w}}+e^{-\frac{|x|^2}{4w}}\right|
=\left|\int_0^1\int_0^1\frac{\partial^2}{\partial \lambda
\partial \mu}e^{-\frac{|x+\lambda h \xi+\mu h \eta|^2}{4w}}d\lambda
d\mu\right|\\
&=&\left|\int_0^1\int_0^1e^{-\frac{|x+\lambda h \xi+\mu h
\eta|^2}{4w}}\left(\frac{1}{4w^2}\langle h\xi, x+\lambda h \xi+\mu h
\eta\rangle\langle h\eta, x+\lambda h \xi+\mu h \eta\rangle
-\frac{1}{2w}\langle h\eta, h \xi\rangle\right)d\lambda
d\mu\right|\\
&\leq&\int_0^1\int_0^1\left(e^{-\frac{|x+\lambda h \xi+\mu h
\eta|^2}{4w}}\frac{1}{4w^2}|x+\lambda h \xi+\mu h \eta|^2
h^2+e^{-\frac{|x+\lambda h \xi+\mu h
\eta|^2}{4w}}\frac{1}{2w}h^2\right)d\lambda d\mu\\
&\leq&C \int_0^1\int_0^1 e^{-\frac{|x+\lambda h \xi+\mu
h|^2}{8w}}\frac{h^2}{w}d\lambda d\mu.
\end{eqnarray*}
By considering the cases $\frac{h^2}{w}\leq 1$ and $\frac{h^2}{w}>
1$, we have for any $0\le \rho_2  \le 2$,
\begin{eqnarray*}
&&\left|e^{-\frac{|x+h\xi+h\eta|^2}{4w}}-e^{-\frac{|x+h\xi|^2}{4w}}-e^{-\frac{|x+h\eta|^2}{4w}}+e^{-\frac{|x|^2}{4w}}\right|\\
&\leq&C
(\frac{h^2}{w})^{\frac{\rho_2}{2}}\int_0^1\int_0^1\left(e^{-\frac{|x+\lambda
h \xi+\mu
h\eta|^2}{8w}}+e^{-\frac{|x+h\xi+h\eta|^2}{4w}}+e^{-\frac{|x+h\xi|^2}{4w}}+e^{-\frac{|x+h\eta|^2}{4w}}+e^{-\frac{|x|^2}{4w}}\right)d\lambda
d\mu\\
&:= &Ch^{\rho_2} w^{ -\frac {\rho2}2} q_{h,\xi,\eta}(x,w).
\end{eqnarray*}
Therefore, we obtain
\begin{eqnarray*}
|f(x+h\xi+h\eta)-f(x+h\xi)-f(x+h\eta)+f(x)|=
C h^{\rho_2}\int_0^{\infty}w^{\frac{\alpha-5-\rho_2 }{2}}e^{-w}q_{h,\xi,\eta}(x,w)dw,
\end{eqnarray*}
and
\begin{eqnarray*}
&&\int_0^T
\int_{S^2}\int_{S^2}\left|f\left(s(\xi+\eta)+h(\xi+\eta)\right)-f\left(s(\xi+\eta)+h\xi\right)-f\left(s(\xi+\eta)+h\eta\right)+f\left(s(\xi+\eta)\right)\right|\\
&& \times s^2 \sigma(d\xi)\sigma(d\eta)ds\\
&\leq&Ch^{\rho_2}\int_0^1\int_0^1\int_0^{\infty}w^{\frac{\alpha-5-\rho_2}{2}}e^{-w}\int_0^T
\int_{S^2}\int_{S^2}\Big(e^{-\frac{|(s+\lambda h)\xi+(s+\mu h
)\eta|^2}{8w}}\\
&&+e^{-\frac{|(s+h)(\xi+\eta)|^2}{4w}}+e^{-\frac{|(s+h)\xi+s\eta|^2}{4w}}+e^{-\frac{|s\xi+(s+h)\eta|^2}{4w}}+e^{-\frac{|s\xi+s\eta|^2}{4w}}\Big)s^2\sigma(d\xi)\sigma(d\eta)ds dwd\lambda
d\mu
\end{eqnarray*}
We claim that when $0<\rho_2< \min(\alpha-1,2)$, the above expression is bounded by $h^{\frac {\rho_2}2}$.
To show this claim, we first estimate the integral
\begin{eqnarray*}
I_2:=\int_0^T \int_{S^2}\int_{S^2}s^2e^{-\frac{|(s+\lambda h)\xi+(s+\mu
h)\eta|^2}{8w}}\si(d\xi)\si(d\eta)ds.
\end{eqnarray*}
Using the Fourier transform (see Lemma \ref{int identity of GGf}) and
the change of variable $\sqrt{w}\xi=\eta$, we obtain
\begin{eqnarray*}
I_2&=&\int_0^T \frac{s^2}{(s+\lambda h)(s+\mu
h)}\int_{\RR^3}w^{\frac{3}{2}}e^{-2w|\xi|^2}\frac{\sin(s+\lambda
h)|\xi|}{|\xi|}\frac{\sin(s+\mu h)|\xi|}{|\xi|}d\xi ds\\
&\leq&C \int_{\RR^3}e^{-2 |\eta|^2}\frac{w}{|\eta|^2}d\eta\leq C
w,
\end{eqnarray*}
where the constant $C$ does not depend on $\lambda$ and $\mu$.
The same estimation can be done for each of the other integrals  and we obtain
\begin{equation*}
\int_0^T
\int_{S^2}\int_{S^2}s^2\left(e^{-\frac{|(s+h)(\xi+\eta)|^2}{4w}}+e^{-\frac{|(s+h)\xi+s\eta|^2}{4w}}+e^{-\frac{|s\xi+(s+h)\eta|^2}{4w}}+e^{-\frac{|s\xi+s\eta|^2}{4w}}\right)\si(d\xi)\si(d\eta)ds\leq
C w.
\end{equation*}
Thus,
\begin{eqnarray*}
&&\int_0^T
\int_{S^2}\int_{S^2}\left|f\left(s(\xi+\eta)+h(\xi+\eta)\right)-f\left(s(\xi+\eta)+h\xi\right)-f\left(s(\xi+\eta)+h\eta\right)+f\left(s(\xi+\eta)\right)\right|\\
&& \times s^2 \sigma(d\xi)\sigma(d\eta)ds\\
&\leq& C h^{\rho_2}\int_0^{\infty}
w^{\frac{\alpha-3-\rho_2 }{2}}e^{-w}dw \leq C
h^{\rho_2},
\end{eqnarray*}
and the \eref{time cond 2} is satisfied for $0<\rho_2 < \min
(\alpha-1,2)$. This completes the proof.
\end{proof}

\medskip
Notice  that, with the notation of Corollary 4.2, for the Bessel kernel we can take $\kappa_1=\kappa_2 <(\alpha -1)\wedge 1$, and we deduce  the local H\"older continuity  of the solution $u$ in space and time  variables of order $\kappa < \min(\gamma_1, \gamma_2,  \frac{\alpha-1}2\wedge 1 )$.

\section{The Fractional Noise}

In this section we consider the case where  $\dot
{W}(t,x)$ is fractional Brownian noise  in the space variable  with Hurst parameters $H_1,H_2,H_3$ in each direction. That is,  suppose that $\left\{ W(t, x),  t\ge 0 , x\in
\mathbb{R}^3\right\}$  is a centered Gaussian field with the
 covariance
\begin{equation}
\EE \left[ W(s, x)W(t, y)\right]= (s\wedge t) \prod_{i=1}^3  R_i(x_i,
y_i)\,, \qquad s\,, t\ge 0\,, x, y\in \RR^3\,,
\end{equation}
where
\begin{equation}
R_i(u , v )=\frac12 \left( |u |^{2H_i}+ |v
|^{2H_i}-|u-v|^{2H_i}\right).
\end{equation}
Then $\dot {W}(t,x)$ is the formal partial derivative $\frac{\partial^4 W}{\partial t \partial x_1 \partial x_2 \partial x_3}(t,x)$.
We will require  $\frac{1}{2}<H_i<1$, $i=1,2,3$. This choice of noise corresponds to the covariance function
\begin{equation}\label{fractional ker}
f(x)=c_H|x_1|^{2H_1-2}|x_2|^{2H_2-2}|x_3|^{2H_3-2},
\end{equation}
where $H=(H_1,H_2,H_3)$ and $c_H=\prod_{i=1}^3 H_i (2H_i-1) $.
Here and in what follows for simplicity, we  omit the coefficient $c_H$ in the expression of $f(x)$.  The corresponding
spectral measure is
\begin{equation}
\mu(d\xi)=C_H|\xi_1|^{1-2H_1}|\xi_2|^{1-2H_2}|\xi_3|^{1-2H_3}
\end{equation}
 for some constant $C_H$ which depends only on $H$. We
will apply Theorems \ref{space Holder Thm} and \ref{time holder Thm}
to get the H\"older continuity  of the solution to Equation  (1.1) in the space and time variables.

\begin{theorem}\label{fractional example}
Assume conditions $(a)$ and $(b)$ in Theorem \ref{space Holder Thm
no Fourier} and let $f$ be given by \eref{fractional ker} (without the constant $c_H$) with
$H_1+H_2+H_3>2$.  Set
\begin{equation} \label{kap}
\bar{\kappa}= H_1+H_2+H_3 -2,
\end{equation}
 and choose  constants $\kappa_i>0$, $i=0,1,2,3$ such that
$\kappa_i < \min(H_i-\frac{1}{2},
\bar{\kappa},\gamma_1,\gamma_2)$ for $ i=1,2,3$, and
$\kappa_0 \le \min(\kappa_1, \kappa_2, \kappa_3)$.
Then  the solution to \eref{our SPDE} is locally H\"older
continuous with exponent $\kappa_0$ in the time variable and with exponent $\kappa_i$ in the $i$th direction. Namely, for any bounded rectangle $I\subset \RR^3$,
there exists a random variable $K$ (depending on $I$ and the constants $\kappa_i$'s), such that \begin{equation*}
|u(t,x)-u(\bar{t},y)|\leq K
(|x_1-y_1|^{\kappa_1}+|x_2-y_2|^{\kappa_2}+|x_3-y_3|^{\kappa_3}+|\bar{t}-t|^{\kappa_0})
\end{equation*}
for all $t,\bar{t} \in
[0,T]$, $x,y \in I$.
\end{theorem}

\begin{proof}
First we consider the space variable.
Proceeding as in the proof of Theorem \ref{space Holder Thm}, it is easy to see that if for
some number $0<\gamma\leq 1$,
$\mathcal{F}\left(|\xi_1|^{2\gamma}\mu(d\xi)\right)(w)$ is a
nonnegative locally integrable function and
\begin{equation}\label{xi_1 int cond}
\int_{\RR^3}\frac{|\xi_1|^{2\gamma}\mu(d\xi)}{1+|\xi|^2}<\infty\,,
\end{equation}
then if  $\kappa_1 =\min(\gamma,\gamma_1,\gamma_2)$,  for any bounded rectangle $I\subset \RR^3$, and for any $q\ge 2$, there exists a constant $C$ such that
\begin{equation} \label{qq1}
\EE |u(t,x_1,x_2,x_3)-u(t,y_1,x_2,x_3)|^q\leq C |x_1-y_1|^{q\kappa_1}
\end{equation}
for any $t\in [0,T]$ and $x,y \in I$.

We claim that for $0<\gamma <
\min(H_1-\frac{1}{2}, \bar{\kappa})$,
$\mathcal{F}\left(|\xi_1|^{2\gamma}\mu(d\xi)\right)(w)$ is a
nonnegative locally integrable function and \eref{xi_1 int cond}
holds. Indeed, since
$\mu(d\xi)=|\xi_1|^{1-2H_1}|\xi_2|^{1-2H_2}|\xi_3|^{1-2H_3}d\xi$, we
have
\begin{eqnarray*}
\mathcal{F}\left(|\xi_1|^{2\gamma}\mu(d\xi)\right)(w)&=&\mathcal{F}\left(|\xi_1|^{1-2H_1+2\gamma}|\xi_2|^{1-2H_2}|\xi_3|^{1-2H_3}\right)(w)\\
&=&C|w_1|^{-2+2H_1-2\gamma}|w_2|^{2H_2-2}|w_3|^{2H_3-2}\,,
\end{eqnarray*}
which is well defined because
$\gamma<H_1-\frac{1}{2}$. To show \eref{xi_1 int cond}, we have
\begin{eqnarray*}
\int_{\RR^3}\frac{|\xi_1|^{2\gamma}\mu(d\xi)}{1+|\xi|^2}&=&\int_{\RR^3}\frac{|\xi_1|^{1-2(H_1-\gamma)}|\xi_2|^{1-2H_2}|\xi_3|^{1-2H_3}}{1+|\xi|^2}d\xi\\
&\leq&\int_{\RR}\frac{|\xi_1|^{1-2(H_1-\gamma)}}{(1+|\xi_1|^2)^{\alpha_1}}d\xi_1\int_{\RR}\frac{|\xi_2|^{1-2H_2}}{(1+|\xi_2|^2)^{\alpha_2}}d\xi_2\int_{\RR}\frac{|\xi_3|^{1-2H_3}}{(1+|\xi_3|^2)^{\alpha_3}}d\xi_3\,,
\end{eqnarray*}
where the $\alpha_i$'s are positive with $\alpha_1+\alpha_2+\alpha_3=1$.
When $1-2(H_1-\gamma)-2\alpha_1 < -1$, $1-2H_2-2\alpha_2 < -1$ and
$1-2H_3-2\alpha_3 < -1$, the above three integrals are finite. It is
elementary to see such $\alpha_i$'s exist under the condition $\gamma
< H_1+H_2+H_3-2$.    The same argument holds for the other coordinates.

For  the time variable, we will check  conditions (1) and (3)
in Theorem \ref{time holder Thm}.
To see that   condition (1) in Theorem
\ref{time holder Thm} is satisfied for some $0 < \nu \leq 1$, take positive numbers
$\varepsilon_i$, $i=1,2,3$ such that
$\varepsilon_1+\varepsilon_2+\varepsilon_3=1$ and
$2H_i-1-\varepsilon_i>0$ for $i=1,2,3$. Then we have
\begin{eqnarray*}
\int_{|z|\leq h}\frac{f(z)}{|z|}dz &\leq& \int_{|z_1|\leq
h}\frac{|z_1|^{2H_1-2}}{|z_1|^{\varepsilon_1}}dz_1\int_{|z_2|\leq
h}\frac{|z_2|^{2H_2-2}}{|z_2|^{\varepsilon_2}}dz_2\int_{|z_3|\leq
h}\frac{|z_3|^{2H_3-2}}{|z_3|^{\varepsilon_3}}dz_3\\
&\leq&C h^{2H_1-1-\varepsilon_1} h^{2H_2-1-\varepsilon_2}
h^{2H_3-1-\varepsilon_3}=h^{2(H_1+H_2+H_3-2)}.
\end{eqnarray*}
So  condition (1) in Theorem \ref{time holder Thm} is satisfied with
$\nu=\min(2(H_1+H_2+H_3-2),1)$.

To check \eref{time cond 1}, let $x=a( \xi+\eta)+ h\eta$. Then we decompose the difference  $f(x+h\xi)-f(x)$
 into the sum of three terms, each
of them containing an increment in one direction, and we obtain
\begin{eqnarray*}
&& |f(x+h\xi)-f(x)|\\
&=&\left||x_1+h\xi_1|^{2H_1-2}|x_2+h\xi_2|^{2H_2-2}|x_3+h\xi_3|^{2H_3-2}-|x_1|^{2H_1-2}|x_2|^{2H_2-2}|x_3|^{2H_3-2}\right|\\
&\leq&\left||x_1+h\xi_1|^{2H_1-2}-|x_1|^{2H_1-2}\right||x_2+h\xi_2|^{2H_2-2}|x_3+h\xi_3|^{2H_3-2}\\
&&+|x_1|^{2H_1-2}\left||x_2+h\xi_2|^{2H_2-2}-|x_2|^{2H_2-2}\right|  |x_3+h\xi_3|^{2H_3-2}\\
&&+|x_1|^{2H_1-2}|x_2|^{2H_2-2}\left  |  |x_3+h\xi_3|^{2H_3-2}-|x_3|^{2H_3-2}\right|.
\end{eqnarray*}
We claim that for some
$\rho_1 \in (0,1]$,   the integral on
$[0,T]\times S^2 \times S^2$   of  each of these three terms with respect to the measure
$s\sigma(d\xi)\sigma(d\eta)ds$ is bounded by $C h^{\rho_1}$.
To show this claim, we apply \eref{increment a} with $d=1$,
$b=\rho_{1}<\min(2H_1-1, 2H_2-1, 2H_3-1, 2(H_1+H_2+H_3-2))$,
$a=2H_i-\rho_1-1$ and $u=x_i$ to the $i$th summand ($i=1,2,3$) and we get
\begin{eqnarray*}
&& |f(x+h\xi)-f(x)|\\
&\leq&h^{\rho_1}\int_{\RR}dw|x_1-hw|^{2H_1-2-\rho_1}\left||w+\xi_1|^{\rho_1-1}-|w|^{\rho_1-1}\right||x_2+h\xi_2|^{2H_2-2}|x_3+h\xi_3|^{2H_3-2}\\
&&+h^{\rho_1}\int_{\RR}dw|x_2-hw|^{2H_2-2-\rho_1}\left||w+\xi_2|^{\rho_1-1}-|w|^{\rho_1-1}\right||x_1|^{2H_1-2}|x_3+h\xi_3|^{2H_3-2}\\
&&+h^{\rho_1}\int_{\RR}dw|x_3-hw|^{2H_3-2-\rho_1}\left||w+\xi_3|^{\rho_1-1}-|w|^{\rho_1-1}\right||x_1|^{2H_1-2}|x_2|^{2H_2-2}\\
&:=&h^{\rho_1}\left(g_{h,\xi}^1(x)+g_{h,\xi}^2(x)+g_{h,\xi}^3(x)\right).
\end{eqnarray*}
We want to show that for $i=1,2,3$
\begin{equation}\label{g^i_h,xi}
\sup_{0<h\leq 1}\int_0^T \int_{S^2 \times S^2} s
g_{h,\xi}^i(s\xi+(s+h)\eta)\sigma(d\xi)\sigma(d\eta)ds <\infty .
\end{equation}
We will consider only the case $i=1$, the other two terms being similar. By
splitting the integral with respect to $w$ into two parts, one
over $|w|\leq 3$, and  another one over $|w|>3$, just as we did  for the Riesz kernel, we have
\begin{eqnarray*}
&&\int_0^T \int_{S^2\times S^2}
sg^1_{h,\xi}(s\xi+(s+h)\eta)\sigma(d\xi)\sigma(d\eta)ds\\
&\leq&\int_0^T \int_{S^2 \times S^2} s \int_{|w|\leq 3}
|s\xi_1+(s+h)\eta_1-hw|^{2H_1-2-\rho_{1}}|(s+h)\xi_2+(s+h)\eta_2|^{2H_2-2}\\
&&\qquad\times\ |(s+h)\xi_3+(s+h)\eta_3|^{2H_3-2}|w+\xi_1|^{\rho_{1}-1}dw\sigma(d\xi)\sigma(d\eta)ds\\
&&+\int_0^T \int_{S^2 \times S^2} s \int_{|w|\leq 3}
|s\xi_1+(s+h)\eta_1-hw|^{2H_1-2-\rho_{1}}|(s+h)\xi_2+(s+h)\eta_2|^{2H_2-2}\\
&&\qquad\times\ |(s+h)\xi_3+(s+h)\eta_3|^{2H_3-2}|w|^{\rho_{1}-1}dw\sigma(d\xi)\sigma(d\eta)ds\\
&&+\int_0^T \int_{S^2 \times S^2} s \int_{|w|> 3}
|s\xi_1+(s+h)\eta_1-hw|^{2H_1-2-\rho_{1}}|(s+h)\xi_2+(s+h)\eta_2|^{2H_2-2}\\
&&\qquad\times|(s+h)\xi_3+(s+h)\eta_3|^{2H_3-2}\left||w+\xi_1|^{\rho_{1}-1}-|w|^{\rho_{1}-1}\right|dw\sigma(d\xi)\sigma(d\eta)ds\\
&:=&I+II+III.
\end{eqnarray*}
For integral $I$, using the change of variable $w+\xi_1\rightarrow
w$ and the Fourier transform,  we can write
\begin{eqnarray*}
I&\leq& \int_0^T \int_{S^2\times S^2} s \int_{|w|\leq 4}
|(s+h)\xi_1+(s+h)\eta_1-hw|^{2H_1-2-\rho_{1}}|(s+h)\xi_2+(s+h)\eta_2|^{2H_2-2}\\
&&\times|(s+h)\xi_3+(s+h)\eta_3|^{2H_3-2}|w|^{\rho_{1}-1}dw
\sigma(d\eta)\sigma(d\xi)ds\\
&\leq&\int_0^T \int_{|w|\leq 4} dw |w|^{\rho_{1}-1} \sup_{w\in
\RR}\int_{\RR^3}\frac{s}{(s+h)^2}|z_1|^{1-2H_1+\rho_{1}}e^{iz_1hw}|z_2|^{1-2H_2}|z_3|^{1-2H_3}\\
&&\times\left(\frac{\sin(s+h)|z|}{|z|}\right)^2 dz ds.
\end{eqnarray*}
By the change of variable $(s+h)z=x$, the bound $|e^{iz_1hw}|\leq 1$
and direct calculation, we see the integral is finite uniformly in
$0< h\leq 1$.

For the integral $II$ we can write
\begin{eqnarray*}
 II&=&\int_{|w|\leq 3}\int_0^T \int_{S^2 \times S^2} s
|s\xi_1+(s+h)\eta_1-hw|^{2H_1-2-\rho_{1}}|(s+h)\xi_2+(s+h)\eta_2|^{2H_2-2}\\
&&\times|(s+h)\xi_3+(s+h)\eta_3|^{2H_3-2}\sigma(d\xi)\sigma(d\eta)ds|w|^{\rho_1-1}dw\\
&=& \int_{|w|\leq 3} \int_0^T \int_{\RR^3\times
\RR^3}\frac{s}{(s+h)^2}G(s+h,dy)G(s+h,dz)\\
&&\times|\frac{s}{s+h}y_1+z_1-hw|^{2H_1-2-\rho_{1}}|y_2+z_2|^{2H_2-2}|y_3+z_3|^{2H_3-2}ds|w|^{\rho_1-1}dw\\
&=&\int_{|w|\leq 3} \int_0^T \int_{\RR^3}
\frac{s}{(s+h)^2} G(s+h)*G^\psi(s+h)(z)\\
&&\times|z_1-hw|^{2H_1-2-\rho_1}|z_2|^{2H_2-2}|z_3|^{2H_3-2}dzds
|w|^{\rho_1-1}dw,
\end{eqnarray*}
where $\psi(y)=\left( \frac s{s+h}y_1,y_2,y_3\right)$, and  $G^\psi(s+h)$ denotes the image of the measure $G(s+h)$ by the mapping $\psi$.
 Then using the Fourier transform we obtain
\begin{eqnarray*}
II&=&\int_{|w|\leq 3}\int_0^T
\int_{\RR^3}\frac{s}{(s+h)^2}\left(\mathcal{F}\left(G(s+h )*G^\psi(s+h )\right)\right)(\xi)\\
&&\times e^{i\xi_1w} |\xi_1|^{1+\rho_{1}-2H_1}|\xi_2|^{1-2H_2}|\xi_3|^{1-2H_3} d\xi ds|w|^{\rho_1-1}dw\\
&\leq& \int_{|w|\leq 3} \int_0^T \int_{\RR^3}\frac{s}{(s+h)^2}
\left|\frac{\sin(s+h)|\xi|}{|\xi|}\right|\left|\frac
{\sin(s+h)|(\frac{s}{s+h}\xi_1,\xi_2,\xi_3)|}{|(\frac{s}{s+h}\xi_1,\xi_2,
\xi_3)|}\right|\\
&&\times |\xi_1|^{1+\rho_{1}-2H_1}|\xi_2|^{1-2H_2}|\xi_3|^{1-2H_3}d\xi ds |w|^{\rho_1-1}dw \\
&\leq& \int_{|w|\leq 3} \int_0^T \int_{|\xi|\leq
1}|\xi_1|^{1+\rho_{1}-2H_1}|\xi_2|^{1-2H_2}|\xi_3|^{1-2H_3}d\xi
ds |w|^{\rho_1-1}dw \\
&&+ \int_{|w|\leq 3} \int_0^T \int_{|\xi|>
1}\frac{s}{(s+h)^2}\left|\frac{\sin(s+h)|\xi|}{|\xi|}\right|\\
&&\quad \times\frac{1}{\frac{s}{s+h}|\xi|}|\xi_1|^{1+\rho_{1}-2H_1}|\xi_2|^{1-2H_2}|\xi_3|^{1-2H_3}d\xi
ds |w|^{\rho_1-1}dw <\infty,
\end{eqnarray*}
where in the first inequality above we used the fact that
$|e^{-i\xi_1w}|\leq 1$, and in the last inequality we used that fact
that $|\sin x|\leq |x|^{\varepsilon}$ for any $\varepsilon > 0$.   The above integral is finite uniformly in
$w$ and $0< h\leq 1$.

For the
third integral $III$, we can bound
$\left||w+\xi_1|^{\rho_{1}-1}-|w|^{\rho_{1}-1}\right|$ by
$C|w|^{\rho_{1}-2}$ as in the example of the Riesz kernel, and
proceed as in the second integral $II$. Applying the same argument for the
other two terms, we get \eref{g^i_h,xi} with $\rho_{1}\in (0, \min
(2H_1-1,2H_2-1,2H_3-1, 2 \bar{\kappa})$.
Therefore, condition \eref{time cond 1} is satisfied with
$0< \rho_1 < \min(2H_1-1,2H_2-1,2H_3-1, 2 \bar{\kappa})$.

For  condition \eref{time cond 2}, we use the inequality
\begin{eqnarray*}
&&\left|f\left((s+h)\xi+(s+h)\eta\right)-f\left((s+h)\xi+s\eta\right)-f\left((s+h)\eta+s\xi\right)+f\left(s(\xi+\eta)\right)\right|\\
&\leq &
\left|f\left((s+h)\xi+(s+h)\eta\right)-f\left((s+h)\xi+s\eta\right)\right|+\left|f\left((s+h)\eta+s\xi\right)-f\left(s\xi+s\eta\right)\right|.
\end{eqnarray*}
Then we can apply the previous procedure to both terms on the right-hand side and the argument is the same as in the case of condition
\eref{time cond 1}. We conclude that \eref{time cond 2} is satisfied
with $0< \rho_2< \min(2H_1-1,2H_2-1,2H_3-1,
2\bar{\kappa})$.

In summary, we can take   $\nu= \min(2 \bar{\kappa},1)$ and $\rho_1=\rho_2 \in (0,
\min(2H_1-1,2H_2-1,2H_3-1, 2 \bar{\kappa}))$, and Theorem 4.2 together with the moment estimate (\ref{qq1}) leads to the desired H\"older continuity in the space and time variables via an application of Kolmogorov's continuity theorem.
\end{proof}

\medskip
Consider Equation \eref{our SPDE}
with vanishing initial conditions $v_0$, $\bar{v}_0$ and coefficients
 $\sigma \equiv 1$ and  $b\equiv 0$. That means,  we consider
the stochastic wave equation with additive fractional noise
\begin{eqnarray}\label{additive-SPDE}
 \left\{\begin{array}{rcl}
 \left(\frac{\partial ^2}{\partial t^2}-\Delta u\right) (t,x)
&=& \dot W(t,x)   \,,\\ \\
u(0, x)=\frac{\partial u}{\partial t}(0, x)&=& 0.
\end{array}
\right.
\end{eqnarray}
The covariance function of the noise is given by \eref{fractional
ker} with $H_i> \frac{1}{2}$ for $i=1,2,3$ and recall that $\bar{\kappa} =H_1+H_2+H_3-2>0$.

For this equation the solution can be written as
\[
u(t, x)=\int_0^t \int_{\RR^3} G(t-s, x-y) W(ds, dy).
\]
In this case we are going to show that $\bar{\kappa}$ is  the
optimal exponent for the H\"older continuity of the solution $u $ in the space and time variables.
\begin{theorem} \label{additive THM}
Let $u$ be the solution to the stochastic partial differential
equation \eref{additive-SPDE}. Then
\begin{description}
\item{(a)} There are two positive constants $c_1$ and $c_2 $ such that
\begin{equation}
c_1|x-y|^{2\bar\kappa } \le \EE\left(|u(t,x)-u(t,y)|^2\right)
 \le c_2|x-y|^{2\kappa } \label{u-x}
\end{equation}
for all $x,y \in \mathbb{R}^3 $ and $t\in [0,T]$.
\item{(b)}   There are two positive  constants
$c_1$ and $c_2$ such that
\begin{eqnarray}
c_1|\bar{t}-t|^{2\kappa}\le \EE\left(|u(t,x)-u(\bar{t},x)|^2\right)\le
c_2|\bar{t}-t|^{2\kappa}. \label{u-t}
\end{eqnarray}
 for all $t,\bar{t} \in [t_0,T]$ and  $x \in \mathbb{R}^3$.
\end{description}
\end{theorem}

\begin{proof}
For any $x \in \mathbb{R}^3$, set $R(x)=\EE\left(u(t,x)u(t,0)\right)$.
It is easy to see that
\begin{equation*}
\EE\left(|u(t,x)-u(t,y)|^2\right)=2\left(R(0)-R(x-y)\right).
\end{equation*}
Without loss of generality, we may assume that $t=1$ and $y=0$. We
have by Lemma \ref{int identity of GGf}
\begin{eqnarray}
R(0)-R(x)&=&\int_0^1 ds \int_{\mathbb{R}^3}\mu(d\xi)(1-e^{i\xi\cdot
x})|\mathcal{F}G(1-s)(\xi)|^2\nonumber\\
&=&\frac{1}{2}\int_{\mathbb{R}^3}d\xi
|\xi_1|^{1-2H_1}|\xi_2|^{1-2H_2}|\xi_3|^{1-2H_3}\frac{1}{|\xi|^2}\left(1-\cos(\xi\cdot
x)\right)\left(1-\frac{\sin(2|\xi|)}{2|\xi|}\right). \nonumber\\
\label{R}
\end{eqnarray}
The integrand is non-negative.  For clarity we may assume that
$|x_1|\le |x_2|\le |x_3|$.   If $|\xi_3|\ge 1$, then
$1-\frac{\sin(2|\xi|)}{2|\xi|}\geq \frac{1}{2}$. Thus  using  the
change of variable $\xi x_3=\eta$, we have
\begin{eqnarray*}
&&R(0)-R(x)\\
&\geq&\frac{1}{4}\int_{ |\xi_3|\geq 1
}|\xi_1|^{1-2H_1}|\xi_2|^{1-2H_2}|\xi_3|^{1-2H_3}\left(1-\cos(\xi
x )\right)\frac{1}{|\xi|^2}d\xi\\
&=&\frac{1}{4}|x_3|^{2\kappa }\int_{  |\eta_3|\geq
|x_3|}|\eta_1|^{1-2H_1}|\eta_2|^{1-2H_2}|\eta_3|^{1-2H_3}\left[
1-\cos\left( \frac{x_1}{x_3} \eta_1 + \frac{x_2}{x_3} \eta_2
+\eta_3\right)\right] \frac{1}{|\eta|^2}d\eta .
\end{eqnarray*}
If $x$ is in a bounded interval $I$, then there is $L>0$ such that
$|x_3|\le L$. Thus
\begin{eqnarray*}
&&R(0)-R(x)\\
&\geq&  \frac{1}{4}|x_3|^{2\kappa }\int_{  |\eta_3|\geq
L}|\eta_1|^{1-2H_1}|\eta_2|^{1-2H_2}|\eta_3|^{1-2H_3}\left[
1-\cos\left( \frac{x_1}{x_3} \eta_1
+ \frac{x_2}{x_3} \eta_2 +\eta_3\right)\right] \frac{1}{|\eta|^2}d\eta\\
&\geq&  \frac{1}{4}|x_3|^{2\kappa }\inf_{|u_1|, |u_2|\le 1} \int_{
|\eta_3|\geq
L}|\eta_1|^{1-2H_1}|\eta_2|^{1-2H_2}|\eta_3|^{1-2H_3}\left[
1-\cos\left( u_1 \eta_1 + u_2\eta_2 +\eta_3\right)\right]
\frac{1}{|\eta|^2}d\eta.
\end{eqnarray*}
It is easy to see  that
\[
g(u_1, u_2) :=\int_{  |\eta_3|\geq
L}|\eta_1|^{1-2H_1}|\eta_2|^{1-2H_2}|\eta_3|^{1-2H_3}\left[
1-\cos\left( u_1 \eta_1 + u_2\eta_2 +\eta_3\right)\right]
\frac{1}{|\eta|^2}d\eta
\]
is a continuous function of $u_1, u_2$ and for any $u_1$ and $u_2$,
$g(u_1, u_2)$ is positive. Thus
\[
\inf_{|u_1|, |u_2|\le 1} g(u_1, u_2)>0.
\]
This proves the  left-hand side  inequality in \eref{u-x}.

To show the second inequality in \eref{u-x}, we can use the
triangular inequality, and it suffices to show the inequality for
$x=(x_1, 0, 0)$. In this case
\begin{eqnarray*}
R(0)-R(x) &\le &
 \frac{1}{2}\int_{\mathbb{R}^3}d\xi
|\xi_1|^{1-2H_1}|\xi_2|^{1-2H_2}|\xi_3|^{1-2H_3}\frac{1}{|\xi|^2}\left(1-\cos(\xi_1
x_1)\right). \nonumber\\
&\le &
 \frac{1}{2}|x_1|^{2\kappa} \int_{\mathbb{R}^3}d\xi
|\xi_1|^{1-2H_1}|\xi_2|^{1-2H_2}|\xi_3|^{1-2H_3}\frac{1}{|\xi|^2}\left(1-\cos(\xi_1)\right)\\
&=& C|x_1|^{2\kappa}
\end{eqnarray*}
which is the second inequality of \eref{u-x}. Hence,  (a) is proved.

Now we turn to consider (b).  Let  $0\leq t < \bar{t}\leq T$. Then
we have
\begin{equation*}
\EE\left(|u(t,x)-u(\bar{t},x)|^2\right)=2Z_1(t,\bar{t},x)+2Z_2(t,\bar{t},x),
\end{equation*}
where
\begin{eqnarray*}
&&Z_1(t,\bar{t},x)=\EE\left(\int_t^{\bar{t}}\int_{\mathbb{R}^3}G(\bar{t}-s,x-y)W(ds,dy)\right)^2
\end{eqnarray*}
and
\begin{eqnarray*}
&&Z_2(t,\bar{t},x)=\EE\left(\int_0^t\int_{\RR^3}\left(G(\bar{t}-t,x-y)-G(t-s,x-y)\right)W(ds,dy)\right)^2.
\end{eqnarray*}
  Integrating with respect to the
 variable $s$ yields
\begin{eqnarray*}
Z_1(t,\bar{t},x)&=&\EE\left(\int_t^{\bar{t}}\int_{\mathbb{R}^3}G(\bar{t}-s,x-y)W(ds,dy)\right)^2\\
&=&C\int_{t}^{\bar{t}}ds\int_{\mathbb{R}^3}|\mathcal{F}G(\bar{t}-s)(\xi)|^2|\xi_1|^{1-2H_1}|\xi_2|^{1-2H_2}|\xi_3|^{1-2H_3}d\xi\\
&=&C\int_{\mathbb{R}^3}\left((\bar{t}-t)-\frac{\sin\left(2(\bar{t}-t)|\xi|\right)}{2|\xi|}\right)
\frac{1}{|\xi|^2}|\xi_1|^{1-2H_1}|\xi_2|^{1-2H_2}|\xi_3|^{1-2H_3}d\xi.
\end{eqnarray*}
With the change of variable $(\bar{t}-t)\xi \rightarrow \eta$  the
last integral becomes
\begin{eqnarray*}
C(\bar{t}-t)^{2\sum_{i=1}^3-3}\int_{\mathbb{R}^3}
\left(1-\frac{\sin(2|\eta|)}{2|\eta|}\right)\frac{1}{|\eta|^2}|\eta_1|^{1-2H_1}
|\eta_2|^{1-2H_2}|\eta_3|^{1-2H_3}d\eta.
\end{eqnarray*}
Therefore, we have
\begin{equation}
c_1 |\bar{t}-t|^{2\kappa+1}\le Z_1(t,\bar{t},x)\le  c_2
|\bar{t}-t|^{2\kappa+1}. \label{z1}
\end{equation}
The term $Z_2$ is slightly more complicated. A direct integration in the
variable $s$ yields
\begin{eqnarray*}
Z_2(t,\bar{t},x)&=&\int_0^t ds \int_{\mathbb{R}^3} d\xi
|\xi_1|^{1-2H_1}|\xi_2|^{1-2H_2}|\xi_3|^{1-2H_3}\\
&&\times\frac{1}{|\xi|^2}\left(\sin\left((\bar{t}-s)|\xi|\right)-\sin\left((t-s)|\xi|\right)\right)^2\\
&\geq& \int_{{|\xi|\geq(\bar{t}-t)^{-1}} }d\xi
|\xi_1|^{1-2H_1}|\xi_2|^{1-2H_2}|\xi_3|^{1-2H_3}\frac{1}{|\xi|^2} \left(A(t,\bar{t},\xi)+B(t,\bar{t},\xi)\right)\\
&=& tI_1+I_2\,,
\end{eqnarray*}
where
\begin{eqnarray*}
 A(t,\bar{t},\xi)&=& t\left(1-\cos((\bar{t}-t)|\xi|)\right)
 \end{eqnarray*}
and
\begin{eqnarray*}
 B(t,\bar{t},\xi)&=&
 \frac{1}{4|\xi|} \sin(2 (\bar t-t)|\xi|) +\frac1{2|\xi|} \sin ((\bar t-t)|\xi|)-
 \frac{1}{4|\xi|} \sin (2\bar t |\xi|) \\
 &&\qquad -\frac{1}{4|\xi|} \sin (2t|\xi|) -\frac{1}{2|\xi|} \sin((\bar t+t)|\xi|).
\end{eqnarray*}
The  change of variable $(\bar{t}-t)\xi=\eta$ yields
\begin{eqnarray}
I_1 &\ge &\int_{|\xi|\geq(\bar{t}-t)^{-1}}d\xi
|\xi_1|^{1-2H_1}|\xi_2|^{1-2H_2}|\xi_3|^{1-2H_3}\frac{1}{|\xi|^2}\frac{1}{t}A(t,\bar{t},\xi)
\nonumber\\
&\geq&(\bar{t}-t)^{2\kappa }\int_{  |\eta |\geq
1}d\eta|\eta_1|^{1-2H_1}|\eta_2|^{1-2H_2}|\eta_3|^{1-2H_3}\frac{1}{|\eta|^2}
  (1-\cos|\eta|)  \nonumber\\
&\geq& c_1|\bar{t}-t|^{2\kappa }.
\end{eqnarray}
Similarly,
\begin{eqnarray*}
I_2&=&\int_{|\xi|\geq(\bar{t}-t)^{-1}}d\xi
|\xi_1|^{1-2H_1}|\xi_2|^{1-2H_2}|\xi_3|^{1-2H_3}\frac{1}{|\xi|^2}B(t,\bar{t},\xi)\\
&\geq&-c_1 \int_{|\xi|\geq (\bar{t}-t)^{-1}}d\xi
|\xi_1|^{1-2H_1}|\xi_2|^{1-2H_2}|\xi_3|^{1-2H_3}\frac{1}{|\xi|^3} \\
&\geq&-c_1(\bar{t}-t)^{2\kappa +1 }\int_{ |\eta |\geq 1} d\eta
\frac{1}{|\eta|^3}|\eta_1|^{1-2H_1}|\eta_2|^{1-2H_2}|\eta_3|^{1-2H_3}.
\end{eqnarray*}
Therefore,
\begin{equation*}
Z_2(t,\bar{t},x)\geq c_1|\bar{t}-t|^{2\kappa
}-c^{\prime}_2|\bar{t}-t|^{2\kappa+1}\geq c_1^{\prime} |\bar{t}-t|^{2\kappa }
\end{equation*}
when  $|\bar{t}-t|$ is  sufficiently small.  So we conclude that
\begin{equation*}
\EE\left(|u(t,x)-u(\bar{t},x)|^2\right)\geq c_1 |\bar{t}-t|^{2\kappa }.
\end{equation*}
On the other hand, we have
\begin{eqnarray*}
Z_2(t,\bar{t},x) &\le &c_2 \int_0^t ds \int_{|\xi|\le (\bar
t-t)^{-1} } d\xi
|\xi_1|^{1-2H_1}|\xi_2|^{1-2H_2}|\xi_3|^{1-2H_3} (\bar t-t)^2 \\
&&\qquad +c_2 \int_0^t ds \int_{|\xi|\ge  (\bar t-t)^{-1} } d\xi
|\xi_1|^{1-2H_1}|\xi_2|^{1-2H_2}|\xi_3|^{1-2H_3}
\frac{1}{|\xi|^2}.
\end{eqnarray*}
Applying the substitution $\xi (\bar t-t)=\eta$ to both the above
integrals, we see that
\begin{eqnarray*}
Z_2(t,\bar{t},x) &\le & c_2 |\bar{t}-t|^{2\kappa }.
\end{eqnarray*}
Thus   (b) is proved.
\end{proof}

Combining the upper bound in  \eref{u-x} and \eref{u-t},  taking into account that the process $u$ is Gaussian and applying  Kolmogorov continuity criterion,  for any $\delta >0$ and any bounded rectangle  $I\subset \mathbb{R}^3$, there
is a random variable $c_{\delta, I} $ such that almost surely
\[
|u(s, x)-u(t, y)| \le c_{\delta,I} \left( |s-t|^{\kappa-\delta}
+|x-y|^{\kappa-\delta}\right) .
\]
   The first inequalities of \eref{u-x} and \eref{u-t} tell us that
the exponent $\kappa$ is the optimal.
\begin{remark}
Theorem 5.1 in \cite{dalang4} shows that the result obtained in Section 5.2 is optimal. The result in Theorem \eref{additive THM} suggests that the result in Theorem
\ref{fractional example} may not be optimal. To prove the result is
optimal or to find the optimal result needs further research.
\end{remark}

\section{Appendix}
In this section we prove some lemmas used in this paper.
\begin{lemma}\label{conv of G G lemma} For any $s\ge t$
\begin{equation*}
\left(G(s)*G(t)\right)(dx)=\frac{1}{8\pi |x|}{\bf
1}_{[s-t,s+t]}(|x|)dx.
\end{equation*}
\end{lemma}
\begin{proof}
To calculate $\left(G(t)*G(s)\right)(dx)$, let us consider two
independent random variables $X$ and $Y$ uniformly distributed on
the spheres with radii $s$ and $t$ respectively with $s\geq t$. Note that the distribution of $X+Y$ is rotationally
invariant. Consider a bounded continuous function $\varphi$ on
$\RR$. We have
\begin{eqnarray*}
&&\EE\left(\varphi(|X+Y|)\right)=\frac{1}{(4\pi)^2}\int_{S^2}\int_{
S^2}\varphi(|sx+ty|)\sigma(dy)\sigma(dx).
\end{eqnarray*}
It is easy to see that in the above expression, the integral with
respect to $\sigma(dy)$ does not depend on $x$, so we can take
$x=x_0=(0,0,1)$, and using spherical coordinates
$y=(\sin\phi\cos\theta,\sin\phi\sin\theta,\cos\phi)$, we have
\begin{eqnarray*}
&&\EE\left(\varphi(|X+Y|)\right)=\frac{1}{4\pi}\int_{S^2}\varphi\left(|s
x_0+ty|\right)\sigma(dy)=\frac{1}{2}\int_0^{\pi}\varphi\left(\sqrt{s^2+t^2+2ts\cos\phi}\right)\sin\phi
d\phi.
\end{eqnarray*}
Making the change of variable $u=\sqrt{s^2+t^2+2ts\cos\phi}$, we
obtain
\begin{eqnarray*}
&&\EE\left(\varphi(|X+Y|)\right)=\frac{1}{2ts}\int_{s-t}^{s+t}\varphi(u)u
du =\frac{1}{2ts}\int_{B(0,s+t)\setminus
B(0,s-t)}\varphi(|z|)\frac{1}{4\pi|z|}dz,
\end{eqnarray*}
where $B(0,r)$ is the ball in $\RR^3$ with center $0$ and radius
$r$. So we conclude that the random variable $X+Y$ has  a density given
by
\begin{equation}
\rho(z)=\frac{1}{8\pi ts |z|}\mathbf{1}_{[s-t,s+t]}(|z|).
\end{equation}
This completes the proof of the lemma.
\end{proof}

\begin{lemma}\label{int identity of GGf}
Let $G(t,dx)$ denote the fundamental solution of the 3 dimensional
wave equation, $f$ is a non-negative and non-negative definite function which satisfies \eref{H1} and $\mu(d\xi)$ is the spectral measure.
Then for fixed
$t,s>0$, we have
\begin{equation}
\int_{\RR^3}\int_{\RR^3}f(x-y)G(s,dx)G(t,dy)=\int_{\RR^3}\frac{\sin
t |\xi|}{|\xi|}\frac{\sin s |\xi|}{|\xi|}\mu(d\xi).
\end{equation}
\end{lemma}
\begin{proof}
Assume $s\geq t$. By Lemma \ref{conv of G G lemma}, we have
\begin{eqnarray*}
&&\int_{\RR^3}\int_{\RR^3}f(x-y)G(s,dx)G(t,dy)=\frac{1}{8\pi}\int_{\RR^3}\frac{1}{|z|}{\bf{1}}_{[s-t,s+t]}(|z|)f(z)dz.
\end{eqnarray*}
Let $\phi(x)=C \exp(\frac{1}{|x|^2-1}){\bf{1}}_{[0,1)}(|x|)$, where
$C$ is a normalization coefficient such that
$\int_{\RR^3}\phi(x)dx=1$, set
$\phi_{\varepsilon}(x)=\frac{1}{\varepsilon^3}\phi(\frac{x}{\varepsilon})$,
with $\varepsilon \leq t$. If we can show that
\begin{equation}
\lim_{\varepsilon \to
0}\int_{\RR^3}\left((\frac{1}{|\cdot|}{\bf{1}}_{[s-t,s+t]}(|\cdot|))*\phi_{\varepsilon}\right)(z)f(z)dz=\int_{\RR^3}\frac{1}{|z|}{\bf{1}}_{[s-t,s+t]}(|z|)f(z)dz\,,\label{GGf
lemma limit}
\end{equation}
then an application of Fourier transform and dominated convergence
theorem leads to the proof. To show \eref{GGf lemma limit}, first we
note that the function
\begin{equation*}
\left((\frac{1}{|\cdot|}{\bf{1}}_{[s-t,s+t]}(|\cdot|))*\phi_{\varepsilon}\right)(x)
\end{equation*}
is supported within a ball centered at the origin with radius $3s$
for every $\varepsilon \leq t$ and it converges to
$\frac{1}{|z|}{\bf{1}}_{[s-t,s+t]}$ almost everywhere. Next, for
$|x|\leq 3s$ we have
\begin{eqnarray*}
\left((\frac{1}{|\cdot|}{\bf{1}}_{[s-t,s+t]}(|\cdot|))*\phi_{\varepsilon}\right)(x)&=&\int_{\RR^3}\frac{1}{|x-z|}{\bf{1}}_{[s-t,s+t]}(|x-z|)\phi_{\varepsilon}(z)dz\\
&\leq&\int_{|x-z|\geq
\frac{|x|}{2}}\frac{1}{|x-z|}\phi_{\epsilon}(z)dz+\int_{|x-z|<\frac{|x|}{2}}\frac{1}{|x-z|}\phi_{\epsilon}(z)dz\\
&=&\frac{2}{|x|}+\int_{|z|<\frac{|x|}{2}}\frac{1}{|z|}\phi_{\varepsilon}(z+x)dz\,,
\end{eqnarray*}
where in the second integral we have used the change of variable
$z-x \rightarrow z$. Since in the second integral $|z|<
\frac{|x|}{2}$, we have $|z+x|\geq |x|-|z|\geq \frac{|x|}{2}$, and
\begin{eqnarray*}
\int_{|z|<\frac{|x|}{2}}\frac{1}{|z|}\phi_{\varepsilon}(z+x)\leq
\int_{|z|<\frac{|x|}{2}}\frac{1}{|z|}\phi_{\varepsilon}(\frac{x}{2})dz=C
|x|^2 \phi_{\varepsilon}(\frac{x}{2})\leq C \frac{1}{|x|}\,,
\end{eqnarray*}
where in the last inequality we used the fact that $\sup_{x \in
\RR^3} |x|^3 \phi(x) < \infty$. So we conclude that
\begin{equation*}
\left((\frac{1}{|\cdot|}{\bf{1}}_{[s-t,s+t]}(|\cdot|))*\phi_{\varepsilon}\right)(x)\leq
\frac{C}{|x|}{\bf{1}}_{[0,3s]}(|x|).
\end{equation*}
Then if \eref{H1} holds an application of dominated convergence
theorem gives \eref{GGf lemma limit}.
\end{proof}
\begin{lemma}\label{varphi G Fourier id}
Let $\varphi$ be a bounded Borel measurable function and assume that \eref{H1} holds. Then we have
\begin{equation}
\int_{\RR^3}\int_{\RR^3}\varphi(x)G(t,dx)\varphi(y)G(t,dy)f(x-y)=\int_{\RR^3}|\mathcal{F}\left(\varphi
G(t)\right)(\xi)|^2\mu(d\xi).
\end{equation}
\end{lemma}
\begin{proof}
 Let $\phi$ and $\phi_{\varepsilon}$  be as  in the proof of Lemma
\ref{int identity of GGf}. Then using the Fourier transform we have
\begin{equation}  \label{eu1}
\int_{\RR^3} \left(\left(\varphi G(t)*
\widetilde{\varphi G(t)}\right)* \phi_{\varepsilon}\right)(x) f(x)dx
  = \int_{\RR^3}\left|\mathcal{F}\left( \varphi
G(t)\right)(\xi)\right|^2
(\mathcal{F}\phi)(\varepsilon\xi)\mu(d\xi),
\end{equation}
where $\widetilde{\varphi G(t)}(x) =\varphi(-x) G(t,-dx)$.
Since $\varphi$ is bounded,  the same argument as in the proof of
Lemma \ref{int identity of GGf} shows that
\begin{equation}   \label{eu2}
\left|\left(\left(\varphi G(t)* \widetilde{\varphi G(t)}\right)*
\phi_{\varepsilon}\right)(x)\right|\leq
\frac{C}{|x|}{\bf{1}}_{[0,3t]}(|x|).
\end{equation}
Then if \eref{H1} holds, an application of the dominated convergence
theorem yields
\begin{eqnarray*}
\lim_{\varepsilon \to 0}\int_{\RR^3} \left(\left(\varphi G(t)*
\widetilde{\varphi G(t)}\right)* \phi_{\varepsilon}\right)(x)
f(x)dx=\int_{\RR^3}\int_{\RR^3}\varphi(x)G(t,dx)\varphi(y)G(t,dy)f(x-y).
\end{eqnarray*}
On the other hand, the estimate  (\ref{eu2}) implies that the quantity in (\ref{eu1}) is uniformly bounded in $\varepsilon$.
Hence, by Fatou's lemma  $\int_{\RR^3}\left|\mathcal{F}\left( \varphi
G(t)\right)(\xi)\right|^2 \mu(d\xi) <\infty$, and by dominated convergence, the right-hand side of   (\ref{eu1}) converges to
$\int_{\RR^3}\left|\mathcal{F}\left( \varphi
G(t)\right)(\xi)\right|^2
  \mu(d\xi)$. This  completes the proof of the lemma.
\end{proof}

More generally, given two measurable and bounded functions $\varphi$ and $\psi$, for any $t>0$ and $w\in \RR^3$  we have
\begin{equation} \label{eu3}
\int_{\RR^3}\int_{\RR^3}\varphi(x)G(t,dx)\psi(y)G(t,dy)f(x-y+w)=\int_{\RR^3} \mathcal{F}\left(\varphi
G(t)\right)(\xi)  \overline{ \mathcal{F} \left(\psi G(t) \right) (\xi) }e^{-i w\cdot \xi}\mu(d\xi).
\end{equation}

\end{document}